\documentclass[12pt]{amsart}

\usepackage{amsfonts}
\usepackage{amsmath}
\usepackage{amssymb}
\usepackage{mathrsfs}
\usepackage{amscd}
\usepackage[all]{xy}
\usepackage[pdftex,final]{graphicx}

\usepackage{hyperref}

\usepackage{exscale,txfonts}

\newtheorem{lem}{Lemma}[section]
\newtheorem{thm}[lem]{Theorem}

\newtheorem{cor}[lem]{Corollary}
\theoremstyle{definition}

\newtheorem{rem}[lem]{Remark}

\newcommand{\F}[1]{\Bbb{F}_{#1}}

\newcommand{\Z}{\Bbb{Z}}

\newcommand{\C}{\Bbb{C}}

\newcommand{\spl}[2]{\mathrm{SL}_{#1}(#2)}
\newcommand{\pspl}[2]{\mathrm{PSL}_{#1}(#2)}
\newcommand{\gl}[2]{\mathrm{GL}_{#1}(#2)}
\newcommand{\pgl}[2]{\mathrm{PGL}_{#1}(#2)}

\newcommand{\pb}[1]{\mathcal{P}(#1)}

\newcommand{\rpb}[1]{{\mathcal{RP}}(#1)}
\newcommand{\rpbker}[1]{\mathcal{RP}_1(#1)}

\newcommand{\bl}[1]{\mathcal{B}(#1)}
\newcommand{\rbl}[1]{{\mathcal{RB}}(#1)}
\newcommand{\ppb}[1]{\mathcal{A}(#1)} 
\newcommand{\gpb}[1]{\left[ #1\right]} 
\newcommand{\sus}[1]{\left\{ #1\right\} }

\newcommand{\bconst}[1]{C_{#1}}

\newcommand{\rcr}{\mathrm{cr}}
\newcommand{\bw}{\Lambda}

\newcommand{\rr}[1]{\mathcal{R}_{#1}}
\newcommand{\dilog}[1]{\mathcal{L}_{#1}}
\newcommand{\pn}[2]{\mathbb{P}^{#1}(#2)}
\newcommand{\projl}[1]{\pn{1}{#1}}

\newcommand{\sq}[1]{G_{#1}}

\newcommand{\<}{\langle}
\renewcommand{\>}{\rangle}

\newcommand{\psyl}[2]{{#1}_{(#2)}}

\newcommand{\id}[1]{\mathrm{Id}_{#1}}

\renewcommand{\ker}[1]{\mathrm{Ker}(#1)}
\newcommand{\image}[1]{\mathrm{Im}(#1)}

\newcommand{\coker}[1]{\mathrm{Coker}(#1)}

\newcommand{\Aut}[2]{\mathrm{Aut}_{#1}(#2)}

\newcommand{\tor}[4]{\mathrm{Tor}^{#1}_{#2}(#3,#4)}

\newcommand{\sgr}[1]{\mathrm{R}_{#1}}
\newcommand{\an}[1]{\left\langle{#1}\right\rangle}
\newcommand{\pf}[1]{\left\langle\!\left\langle{#1}\right\rangle\!\right\rangle}
\newcommand{\gw}[1]{\mathrm{GW}(#1)}
\newcommand{\aug}[1]{\mathcal{I}_{#1}}
\newcommand{\gwaug}[1]{\mathrm{I}(#1)}
\newcommand{\gwrel}[1]{\mathcal{J}_{#1}}

\newcommand{\igr}[1]{\Z[#1]}

\newcommand{\matr}[4]{\left[\begin{array}{cc}
#1&#2\\
#3&#4\\
\end{array}
\right]}

\newcommand{\extpow}[2]{{#2}\wedge{#2}}
\newcommand{\Extpow}[3]{\wedge^{#1}_{#2}(#3)}
\newcommand{\sextpow}[3]{\mathrm{SE}^{#1}_{#2}(#3)}
\newcommand{\tens}[3]{\mathrm{T}^{#1}_{#2}(#3)}
\newcommand{\sym}[3]{\mathrm{Sym}^{#1}_{#2}(#3)}
\newcommand{\symm}{\ast}
\newcommand{\asymm}{\circ}
\newcommand{\asym}[3]{\mathrm{S}^{#1}_{#2}(#3)}
\newcommand{\rasym}[3]{\mathrm{RS}^{#1}_{#2}(#3)}
\newcommand{\rsf}[2]{[#1,#2]}
\newcommand{\zhalf}[1]{#1} 
\newcommand{\zzhalf}[1]{{#1}'}

\newcommand{\Inv}[2]{\mathrm{H}^0\left( #2,#1\right)}

\newcommand{\gal}[2]{\mathrm{Gal}(#1 /#2)}

\newcommand{\mwk}[2]{K^{\mathrm{\small MW}}_{#1}({#2})}

\newcommand{\milk}[2]{K^{\mathrm{\small M}}_{#1}({#2})}
\newcommand{\kind}[1]{K^{\mathrm{\small ind}}_3(#1)}

\newcommand{\ho}[3]{\mathrm{H}_{#1}(#2,#3 )}
\newcommand{\hoz}[2]{\mathrm{H}_{#1}(#2)}

\newcommand{\hinv}[2]{\mathrm{inv}_{#1}{#2}}
\newcommand{\conj}[2]{\mathrm{Conj}(#1,#2)}

\newcommand{\rs}{\mathrm{res}}
\newcommand{\cores}{\mathrm{cor}}

\newcommand{\Ind}[3]{\mathrm{Ind}_{#1}^{#2}#3}

\newcommand{\Tor}[2]{\mathrm{Tor}^{\Z}_{1}(#1,#2)}
\newcommand{\covtor}[1]{\mathrm{Tor}^{\Z}_1\widetilde{(\mu_{#1},\mu_{#1})}}

\title{A Bloch-Wigner complex  for $\mathrm{SL}_2$}
\author{Kevin Hutchinson}
\address{School of Mathematical Sciences,
 University College Dublin}
\email{kevin.hutchinson@ucd.ie}
\date{\today}


\begin{document}
\bibliographystyle{plain}
\maketitle

\begin{abstract}
We introduce a refinement of the Bloch-Wigner  complex of a field $F$. This refinement  is  complex of modules over 
 the multiplicative group of the field. Instead of computing the $K_2(F)$ and $\kind{F}$ - as the classical 
Bloch-Wigner complex does - it calculates the second and third integral homology of $\spl{2}{F}$. On passing 
to $F^\times$-coinvariants we recover the classical Bloch-Wigner complex. We include the case of finite fields
throughout the article. 

\end{abstract}

\section{Introduction}\label{sec:intro}
 
What is now usually referred to the as the \emph{Bloch group} of a field $F$ arose first in the work of S. Bloch 
as an explicitly-presented approximation to indecomposable $K_3$ of the 
field  which could be used to define a regulator map based on the dilogarithm (see the notes 
\cite{bloch:regulators}).  When $F=\C$ (and, more generally,
when $F^\times=(F^\times)^2$) there is a natural identification $\kind{\C}=\ho{3}{\spl{2}{\C}}{\Z}$, and this latter group
is a natural target for invariants of hyperbolic $3$-manifolds. It was because of this connection with hyperbolic 
geometry that Dupont and Sah (\cite{sah:dupont} and \cite{sah:discrete3}) explored the properties of the Bloch group. 
In particular,  they wrote down a proof of the so-called \emph{Bloch-Wigner Theorem} 
(\cite{sah:dupont}, Theorem 4.10): The pre-Bloch group (or \emph{scissors congruence group}) 
of the field $F$ is the group, $\pb{F}$, with generators 
$\gpb{x}$, $x\in F^\times\setminus\{ 1\}$ subject to the relations
\[
R_{x,y}:\quad  \gpb{x}-\gpb{y}+\gpb{y/x}-\gpb{(1-x^{-1})/(1-y^{-1})}+\gpb{(1-x)/(1-y)} \quad x\not=y.
\]  
(These relations correspond to the $5$-term functional equation satisfied by the classical dilogarithm. See Zagier
\cite{zagier:dilog} for a beautiful exposition of these and related matters.) We will let $\asym{2}{\Z}{F^\times}$ 
denote the antisymmetric product
\[
\frac{F^\times\otimes_{\Z}F^\times}{<x\otimes y + y\otimes x | x,y \in F^\times>}.
\]
 Then there is a well-defined group homomorphism
\[
\lambda:\pb{F}\to\asym{2}{\Z}{F^\times},\quad \gpb{x}\mapsto (1-x)\otimes x
\] 
and the theorem of Bloch and Wigner says that there is an exact sequence
\[
\xymatrix{
0\ar[r]
&\mu_\C\ar[r]
&\ho{3}{\spl{2}{\C}}{\Z}\ar[r]
&\pb{\C}\ar[r]^-{\lambda}
&\asym{2}{\Z}{\C}\ar[r]
&\ho{2}{\spl{2}{\C}}{\Z}
\ar[r]
&0.}
\]
The argument of Dupont and Sah works equally well for any algebraically closed field and more generally 
for any quadratically closed field (i.e. satisfying $F^\times=(F^\times)^2$).  When the field $F$ is quadratically 
closed then the homology groups can be interpreted in terms of $K$-theory: $\ho{3}{\spl{2}{F}}{\Z}=\kind{F}$ and 
$\ho{2}{\spl{2}{F}}{\Z}=K_2(F)=\milk{2}{F}$. Thus the homology groups of the \emph{Bloch-Wigner complex}
\[
\xymatrix{\pb{F}\ar[r]^-{\lambda}
&\asym{2}{\Z}{F}
}
\] 
are (essentially) the $K$-theory groups $\kind{F}$ and $\milk{2}{F}$. The group $\bl{F}=\ker{\lambda}$ is the 
\emph{Bloch group} of the $F$.

Suslin showed that, interpreted in this 
way, the Bloch-Wigner theorem extends to all (infinite) fields. He proved (see \cite{sus:bloch}, Theorem 5.2) 
that for any infinite field 
$F$ there is a natural exact sequence
\[
\xymatrix{
0\ar[r]
&\covtor{F}\ar[r]
&\kind{F}\ar[r]
&\pb{F}\ar[r]^-{\lambda}
&\asym{2}{\Z}{F}\ar[r]
&\milk{2}{F}
\ar[r]
&0.}
\]
where $\covtor{F}$ is the unique nontrivial extension of $\Tor{\mu_F}{\mu_F}$ by $\Z/2$ if the characteristic 
of $F$ is not $2$, and $\covtor{F}=\Tor{\mu_F}{\mu_F}$ in characteristic $2$.

The purpose of the current article is to extend the original sequence of Bloch-Wigner-Dupont-Sah in another direction:
 namely, to construct a complex, which coincides with the one above when $F$ is quadratically closed, but which 
calculates in the general case - instead of $K$-theory - the homology groups 
$\ho{3}{\spl{2}{F}}{\Z}$ and $\ho{2}{\spl{2}{F}}{\Z}$. Our main goal is to understand better the structure of the 
unstable homology group $\ho{3}{\spl{2}{F}}{\Z}$ and its relation to $\kind{F}$. 

To put this project in context, we recall some of what is known about the relationship between the homology groups 
and the $K$-theory groups. In general, the group extension 
\[
1\to\spl{n}{F}\to\gl{n}{F}\to F^\times\to 1
\]   
defines an action of $F^\times$ on the homology groups $\ho{k}{\spl{n}{F}}{\Z}$. Since the determinant of 
a scalar matrix is an $n$-th power, the subgroup $(F^\times)^n$ acts trivially. In the particular, the 
groups $\ho{k}{\spl{2}{F}}{\Z}$ are modules over the integral group ring $\sgr{F}:=\Z[F^\times/(F^\times)^2]$. The 
natural map $\ho{2}{\spl{2}{F}}{\Z}\to K_2(F)$ (via stabilization and an inverse Hurewicz map) is surjective and 
induces an isomorphism on $F^\times$-coinvariants
\[
\ho{2}{\spl{2}{F}}{\Z}_{F^\times}\cong K_2(F).
\]
However, the action of $F^\times$ on $\ho{2}{\spl{2}{F}}{\Z}$ is in general nontrivial. The action of $\sgr{F}$ 
factors through the Grothendieck-Witt ring $\gw{F}$ of the field, and the kernel of the surjective 
map  $\ho{2}{\spl{2}{F}}{\Z}\to K_2(F)$ is isomorphic, as a $\gw{F}$-module, to $\gwaug{F}^3$, the third power 
of the fundamental ideal $\gwaug{F}$ of the augmented ring $\gw{F}$. (See Suslin \cite{sus:tors}, Appendix for the 
details.) To be more explicit $\ho{2}{\spl{2}{F}}{\Z}$ can be expressed as a fibre product
\[
\ho{2}{\spl{2}{F}}{\Z}\cong \gwaug{F}^2\times_{\gwaug{F}^2/\gwaug{F}^3}\milk{2}{F}.
\]

$\ho{2}{\spl{2}{F}}{\Z}$ is of interest in its own right to $K$-theorists and geometers because it coincides with 
the second Milnor-Witt $K$-group, $\mwk{2}{F}$, of the field $F$ (see, for example, \cite{mor:trieste} or 
\cite{morel:puiss}). More generally, the calculation of the groups $\ho{n}{\spl{n}{F}}{\Z}$, which are at the boundary 
of the homology stability range, involves the Milnor-Witt $K$-groups $\mwk{n}{F}$ (\cite{hutchinson:tao3}). 

The group $\ho{3}{\spl{2}{F}}{\Z}$ is of interest, among other reasons, because it is strictly below the range of 
homology stability.
However there is, for any field $F$, a natural homomorphism $\ho{3}{\spl{2}{F}}{\Z}\to\kind{F}$ which 
induces a surjective homomorphism (see \cite{hutchinson:tao2})
\[
\xymatrix{\ho{3}{\spl{2}{F}}{\Z}_{F^\times}\ar@{-{>>}}[r]& \kind{F}}.
\] 
 Suslin has asked the question whether this is an isomorphism, and it is known (see Mirzaii \cite{mirzaii:third})
that the kernel consists of - at worst - $2$-primary torsion. 

In order to refine the Bloch-Wigner sequence to a sequence which captures the homology of $\spl{2}{F}$, it is 
necessary to build in the $\sgr{F}$-module structures at each stage.
Thus, in this article we introduce first the \emph{refined pre-Bloch group} $\rpb{F}$ of a field $F$. 
This is the $\sgr{F}$-module generated by symbols $\gpb{x}$, $x\in F^\times\setminus\{ 1\}$ subject to the relations
\[
0=[x]-[y]+\an{x}\left[ y/x\right]-\an{x^{-1}-1}\left[(1-x^{-1})/(1-y^{-1})\right]
+\an{1-x}\left[(1-x)/(1-y)\right],\quad x,y\not= 1
\]
where $\an{x}$ denotes the class of $x$ in $F^\times/(F^\times)^2$. Similarly we introduce an $\sgr{F}$-module 
$\rasym{2}{\Z}{F^\times}$ which has natural generators $\rsf{x}{y}$,  $x,y\in F^\times$. The `refined Bloch-Wigner
complex' is then the complex of $\sgr{F}$-modules  
\[
\xymatrix{\rpb{F}\ar[r]^-{\Lambda}&\rasym{2}{\Z}{F}},\quad \gpb{x}\mapsto \rsf{1-x}{x}.
\]
On taking $F^\times$-coinvariants this reduces to the classical Bloch-Wigner complex.  Our main result (Theorem 
\ref{thm:main}) is that there is, for any field $F$, a natural complex of $\sgr{F}$-modules 
\[
\xymatrix{0\ar[r]
&\Tor{\mu_F}{\mu_F}\ar[r]
&\ho{3}{\spl{2}{F}}{\Z}\ar[r]
&\rpb{F}\ar[r]^-{\Lambda}
&\rasym{2}{\Z}{F^\times}\ar[r]
&\ho{2}{\spl{2}{F}}{\Z}\ar[r]
&0}
\]
which is exact at every term except possibly at the term $\ho{3}{\spl{2}{F}}{\Z}$, where the homology of the 
complex is annihilated by $4$.  

The \emph{refined Bloch group} of the field $F$ is the $\sgr{F}$-module $\rbl{F}:=\ker{\Lambda}$. The main theorem 
tells us that it is a good approximation to the $\ho{3}{\spl{2}{F}}{\Z}$. In particular, we show that - up to some 
$2$-primary torsion - $\rbl{F}_{F^\times}\cong\bl{F}$ and 
\[
\ker{\rbl{F}\to\bl{F}}\cong \ker{\ho{3}{\spl{2}{F}}{\Z}\to\kind{F}}.
\]

In a separate article we will develop further the algebraic properties of the refined Bloch group
(see \cite{hut:arxivrbl11}, for example). In particular, when the field $F$ has a valuation with residue field 
$k$, there are useful specialization homomorphisms from $\rbl{F}$ to $\pb{k}$. We will use these maps to show that if 
$F$ is a nondyadic local field with (finite) residue field $k$ there is a natural isomorphism
\[
\ho{3}{\spl{2}{F}}{\Z[1/2]}\cong \left(\kind{F}\oplus \pb{k}\right)\otimes \Z[1/2].
\]

Similarly, we will show (see \cite{hut:arxivrbl11} for details) 
that for any global field  $F$ the kernel 
\[
\ker{\ho{3}{\spl{2}{F}}{\Z[1/2]}\to\kind{F}\otimes\Z[1/2]}
\]
 maps 
homomorphically onto the infinite direct sum 
\[
\left(\oplus_v \pb{k_v}\right)\otimes\Z[1/2],
\]
  the sum being over all finite places $v$ of $F$.

Thus the (refined) Bloch groups of finite fields will play an important role in future applications. 
Because of this, and 
unlike most of the references above, throughout the paper we include the case of finite fields. At times, they 
require separate treatment and methods. For this reason we include a separate section - section \ref{sec:finite}- 
recalling the results  
we need on the homology of $\spl{2}{F}$ for finite fields $F$. In the last section of the paper we combine our 
main theorem with these homology calculations to give a proof of Suslin's theorem in the case of finite fields and to 
make some useful calculations in Bloch groups  of finite fields.     
\begin{rem} Several authors (W. Neumann \cite{neumann:extended}, W. Nahm, S. Goette and C. 
Zickert \cite{zickert:goette}) 
have introduced and studied an 
\emph{extended Bloch group}, which is exactly isomorphic to the $\kind{F}$ - at 
least for some fields $F$. This is a quite different object from the  \emph{refined Bloch group} 
introduced here, which effectively 
bears the same relationship to $\ho{3}{\spl{2}{F}}{\Z}$ as the classical Bloch group does to $\kind{F}$.
\end{rem}

\section{Bloch Groups and the Bloch-Wigner map}\label{sec:bloch}
In this section, we review the definition of the classical Bloch group and pre-Bloch group of a field, and we 
define our basic objects 
of study in this article, the refined 
Bloch group and refined pre-Bloch group.

\subsection{Some notation in this article}
For a field $F$, 
we let $\sq{F}$ denote the multiplicative group, $F^\times/(F^\times)^2$, of nonzero square classes of the field.
For $x\in F^\times$, we will let $\an{x} \in \sq{F}$ denote the corresponding square class. 
Let $\sgr{F}$ denote the integral group 
ring $\igr{\sq{F}}$ of the group $\sq{F}$. 
We will use the notation $\pf{x}$ for the basis elements, $\an{x}-1$, of the augmentation 
ideal $\aug{F}$ of $\sgr{F}$.

For a commutative ring $A$ and an $A$-module $M$, we let $\tens{n}{A}{M}$ denote 
the $n$-fold tensor product of $M$ over $A$. We let $\Extpow{n}{A}{M}$ denote the $n$-th exterior power of 
$M$ over $A$; i.e. the $n$-th term of the graded ring $\tens{n}{A}{M}/I$ where $I$ is the ideal generated 
by the elements $m\otimes m$, $m\in M$. We let $\sym{n}{A}{M}$ denote the $n$-th symmetric power of 
$M$ over $A$; i.e. the $n$-th term of the graded ring $\tens{n}{A}{M}/J$ where $J$ is the ideal generated 
by the elements $m\otimes n-n\otimes m$, $m, n\in M$. 

For any abelian group $A$ we let $\zzhalf{A}$ denote $A\otimes\Z[1/2]$. 

\subsection{The classical Bloch group} 
Let $F$ be a field with at least $4$ elements 
and let $X_n$ denote the set of all ordered $n$-tuples of distinct points of $\projl{F}$. $\pgl{2}{F}$, and hence also 
$\gl{2}{F}$, acts on $\projl{F}$ via fractional linear transformations. 
Thus these groups act on $X_n$ via a diagonal action.  

Now let $\ppb{F}$ be the cokernel of the homomorphism of $\gl{2}{F}$-modules
\[
\delta: \Z X_5 \to \Z X_4,\quad (x_1,\ldots,x_5)\mapsto \sum_{j=1}^5(-1)^{j+1}(x_1,\ldots,\hat{x_j},\ldots,x_5).
\]

 Then the \emph{pre-Bloch group} of $F$ is the group 
\[
\pb{F}:=\ppb{F}_{\gl{2}{F}}=\coker{\bar{\delta}:\left(\Z X_5\right)_{\gl{2}{F}}\to\left(\Z X_4\right)_{\gl{2}{F}}}. 
\]

Now the orbits of 
$\gl{2}{F}$ on $X_4$ are classified by the cross-ratio: i.e., in general, $(x_1,\ldots,x_4)$ is in the orbit of 
$(0,\infty,1,x)$ where $x\in \projl{F}\setminus\{\infty,0,1\}=F^\times\setminus\{ 1\}$ is the cross-ratio
\[
\frac{(x_4-x_1)(x_3-x_2)}{(x_3-x_1)(x_4-x_2)}
\]
 of $x_1,\ldots,x_4$.

Thus
\[
(\Z X_4)_{\gl{2}{F}} \cong \bigoplus_{x\in F^\times\setminus\{ 1\}}\Z\cdot (0,\infty,1,x)
\]   
and, similarly,
\[
(\Z X_5)_{\gl{2}{F}} \cong 
\bigoplus_{\substack{x,y \in F^\times\setminus\{ 1\}\\ x\not=y}} \Z\cdot (0,\infty,1,x,y).
\]

For $x\not= y$ in $F^\times\setminus\{ 1\}$, $\bar{\delta}(0,\infty,1,x,y)$
\begin{eqnarray*}
&=(\infty,1,x,y)-(0,1,x,y)+(0,\infty,x,y)-(0,\infty,1,y)+(0,\infty,1,x)\\
&=\left( 0,\infty,1,\frac{1-x}{1-y}\right)-\left(0,\infty,1,\frac{1-x^{-1}}{1-y^{-1}}\right)
+\left(0,\infty,1,\frac{y}{x}\right)
- (0,\infty,1,y)+(0,\infty,1,x)
\end{eqnarray*}

Thus, if we let $\gpb{x}$ denote the class of the orbit of $(0,\infty,1,x)$ in $\pb{F}$ then $\pb{F}$ is the group 
generated 
by the elements $\gpb{x}$, $x\in F^\times\setminus\{ 1\}$,  subject to the relations 
\[
R_{x,y}:\quad  \gpb{x}-\gpb{y}+\gpb{y/x}-\gpb{(1-x^{-1})/(1-y^{-1})}+\gpb{(1-x)/(1-y)} 
\] 
for $x\not= y$. 

Let $\asym{2}{\Z}{F^\times}$ denote the group 
\[
\frac{F^\times\otimes_{\Z}F^\times}{<x\otimes y + y\otimes x | x,y \in F^\times>}
\]
and denote by $x\asymm y$ the image of $x\otimes y$ in $\asym{2}{\Z}{F^\times}$.

The map 
\[
\lambda:\pb{F}\to \asym{2}{\Z}{F^\times},\quad  [x]\mapsto \left(1-{x}\right)\asymm {x}
\]
is well-defined, and the \emph{Bloch group of $F$}, $\bl{F}\subset \pb{F}$, is defined to be the kernel of $\lambda$. 

\subsection{The refined pre-Bloch group}

Let $F$ be a field with at least $4$ elements. The \emph{refined pre-Bloch group of $F$} is the group 
\[
\rpb{F}:=\ppb{F}_{\spl{2}{F}}=\coker{\bar{\delta}:(\Z X_5)_{\spl{2}{F}}\to(\Z X_4)_{\spl{2}{F}}}.
\]

Since for any field $F$ we have a short exact sequence of groups
\[
1\to\pspl{2}{F}\to\pgl{2}{F}\to\sq{F}\to 1
\]
it follows that if $X$ is any $\pgl{2}{F}$-set, then $\spl{2}{F}\backslash X$ is a $\sq{F}$-set and 
\[
(\Z X)_{\spl{2}{F}}\cong \Z[\spl{2}{F}\backslash X]
\]
is an $\sgr{F}$-module.

The stabilizer in $\spl{2}{F}$ of $(0,\infty)$ is the subgroup $T$ consisting of all diagonal matrices 
\[
D(a):= \left[
\begin{array}{cc}
a&0\\
0&a^{-1}
\end{array}
\right]
\]
For $x\in\projl{F}$, $D(a)\cdot x=a^2x$.

Given $x\not= y$ in $\projl{F}$,  let $T_{x,y}\in\spl{2}{F}$ be the matrix
\[
T_{x,y}=\left\{
\begin{array}{ll}
\matr{1}{-x}{\frac{1}{x-y}}{\frac{-y}{x-y}} & x,y\not= \infty\\
&\\
\matr{1}{-x}{0}{1} & y=\infty\\
&\\
\matr{0}{-1}{1}{-y} & x=\infty
\end{array}\right.
\]
Then $T_{x,y}(x)=0$, $T_{x,y}(y)=\infty$, and, by the preceding remarks, 
if $S\in\spl{2}{F}$ satisfies $S(x)=0$ and $S(y)=\infty$, 
then $S=D(a)\cdot T_{x,y}$ for some $a\in F^\times$. In particular, if $A\in\spl{2}{F}$, it follows that 
\[
T_{Ax,Ay}=D(a)\cdot T_{x,y}\cdot A^{-1}\mbox{ for some }a=a(x,y,A)\in F^\times.
\] 

For $x,y,z$ distinct points of $\projl{F}$, we define
\[
\phi(x,y,z):=T_{x,y}(z)=
\left\{
\begin{array}{ll}
(z-x)(x-y)(z-y)^{-1},& x,y,z\not=\infty\\
(y-z)^{-1},& x=\infty\\
z-x,& y=\infty\\
x-y,&z=\infty
\end{array}
\right.
\]
Thus $\phi(x,y,z)\in \projl{F}\setminus\{\infty,0\} = F^\times$, and $\phi(0,\infty,z)=z$ for $z\in F^\times$. Furthermore,
if $A\in \spl{2}{F}$, then 
\[
\phi(Ax,Ay,Az) = T_{Ax,Ay}(Az)=D(a)\cdot T_{x,y}\cdot A^{-1}(Az) = a^2\phi(x,y,z) \mbox{ for some } a\in F^\times. 
\]

Now, for 
$n\geq 1$,  let $Y_n$ denote the set of ordered $n$-tuples of distinct points of $F^\times$. 
$Y_n$ is an $F^\times$-set via the diagonal 
action. 
\begin{lem}\label{lem:phin}
For $n\geq 3$, the map 
\begin{eqnarray*}
\Phi_n:X_n& \to & Y_{n-2}\\
(x_1,x_2,\ldots,x_n)& \mapsto & (\phi(x_1,x_2,x_3),\phi(x_1,x_2,x_4),\ldots,\phi(x_1,x_2,x_n))
\end{eqnarray*}
induces a bijection of $\sq{F}$-sets 
\[
\spl{2}{F}\backslash X_n\longleftrightarrow (F^\times)^2\backslash Y_{n-2}.
\]
\end{lem}
\begin{proof} By the remarks above $\Phi_n$ descends to a well-defined map\\
 $\bar{\Phi}_n:\spl{2}{F}\backslash X_n\to (F^\times)^2\backslash Y_{n-2}$.   Furthermore, the map 
\begin{eqnarray*}
\Psi_n:Y_{n-2} & \to & X_n\\
(y_1,\ldots,y_{n-2})& \mapsto & (0,\infty,y_1,\ldots, y_{n-2})
\end{eqnarray*}
gives a set-theoretic section of $\Phi_n$ which descends to an inverse of $\bar{\Phi_n}$. 

Since, for any $a\in F^\times$, 
\[
\phi(ax_1,ax_2,ay)=\left\{\begin{array}{ll}
a\phi(x_1,x_2,y),& x_1\not=\infty\\
a^{-1}\phi(x_1,x_2,y),& x_1=\infty\\
\end{array}\right.
\]
it also follows that $\bar{\Phi}_n$ is a map of $\sq{F}$-sets.
\end{proof}

\begin{cor}\label{cor:xnsl}
For $n\geq 0$, let $Z_n$ denote the set of ordered $n$-tuples, $[z_1,\ldots,z_n]$,  of distinct 
points of $F^\times\setminus\{ 1\}$. 
Then for all $n\geq 3$ there is an isomorphism of $\sgr{F}$-modules
\[
(\Z X_n)_{\spl{2}{F}}\cong \sgr{F}[Z_{n-3}].
\]
\end{cor}
\begin{proof}
By Lemma \ref{lem:phin} we have $\sgr{F}$-isomorphisms
\[
(\Z X_n)_{\spl{2}{F}}\cong \Z[\spl{2}{F}\backslash X_n]\cong \Z[(F^\times)^2\backslash Y_{n-2}].
\]
Finally, we have an $\sgr{F}$-isomorphism
\[
\Z[(F^\times)^2\backslash Y_{n-2}]\cong\sgr{F}[Z_{n-3}]
\]
via the map
\[
(y_1,\ldots,y_{n-2})\mapsto \an{y_1}\left[\frac{y_2}{y_1},\ldots,\frac{y_{n-2}}{y_1}\right].
\]
\end{proof}

It follows that the $\sgr{F}$-isomorphism $(\Z X_n)_{\spl{2}{F}}\cong \sgr{F}[Z_{n-3}]$ is given by
\[
(x_1,\ldots,x_n)\mapsto \an{\phi(x_1,x_2,x_3)}\left[\frac{\phi(x_1,x_2,x_4)}{\phi(x_1,x_2,x_3)},\ldots,
\frac{\phi(x_1,x_2,x_n)}{\phi(x_1,x_2,x_3)}\right].
\]

In particular, we have $\sgr{F}$-isomorphisms 
\[
(\Z X_3)_{\spl{2}{F}}\cong \sgr{F},\quad (\Z X_4)_{\spl{2}{F}}\cong \sgr{F}[F^\times\setminus\{ 1\}],\quad 
(\Z X_5)_{\spl{2}{F}}\cong \sgr{F}[Z_2] 
\] 

Note that taking $\sq{F}$-coinvariants of the terms in Corollary \ref{cor:xnsl} we obtain
\begin{cor}\label{cor:xngl}
For all $n\geq 3$ there is an isomorphism of groups
\[
(\Z X_n )_{\gl{2}{F}}\cong \Z [Z_{n-3}].
\]
\end{cor}

In particular, for $n=4$, the isomorphism $(\Z X_4)_{\gl{2}{F}}\cong \Z[F^\times\setminus\{ 1\}]$ is given by 
\[
(x_1,x_2,x_3,x_4)\mapsto \left[\frac{\phi(x_1,x_2,x_4)}{\phi(x_1,x_2,x_3)}\right]
\]
which is just the classical cross-ratio map.

Now it follows from the calculations above that the map
\begin{eqnarray*}
\xymatrix{\sgr{F}[Z_2]\ar[r]^-{\cong}&(\Z X_5)_{\spl{2}{F}}\ar[r]^-{\bar{\delta}}&
(\Z X_4)_{\spl{2}{F}}\ar[r]^-{\cong}&\sgr{F}[Z_1]}
\end{eqnarray*}
is the  $\sgr{F}$-module homomorphism
\begin{eqnarray*}
[x,y]\quad \mapsto \quad [x]-[y]+\an{x}\left[ y/x\right]-\an{x^{-1}-1}\left[(1-x^{-1})/(1-y^{-1}\right]
+\an{1-x}\left[(1-x)/(1-y)\right]
\end{eqnarray*}
(since  $\phi(\infty,1,a)=(1-a)^{-1}$, $\phi(0,1,a)=(a^{-1}-1)^{-1}$ and $\phi(0,\infty,a)=a$.)
Thus we have:
\begin{lem}\label{lem:pbpres}
The refined pre-Bloch group $\rpb{F}$ is the $\sgr{F}$-module with generators $\gpb{x}$, $x\in F^\times$ 
subject to the relations $\gpb{1}=0$ and 
\[
S_{x,y}:\quad 0=[x]-[y]+\an{x}\left[ y/x\right]-\an{x^{-1}-1}\left[(1-x^{-1})/(1-y^{-1})\right]
+\an{1-x}\left[(1-x)/(1-y)\right],\quad x,y\not= 1
\]
\end{lem}

Of course, by definition, we have $\pb{F}=(\rpb{F})_{F^\times}=\ho{0}{F^\times}{\rpb{F}}$.

\subsection{The module $\rasym{2}{\Z}{F^\times}$ and the refined Bloch group of a field}

\begin{lem}\label{lem:sym}
Let $G$ be an abelian group. There is a natural short exact sequence of $\igr{G}$-modules
\[
0\to\aug{G}^3\to\aug{G}^2\to\sym{2}{\Z}{G}\to 0
\] 
(where $G$ acts trivially on the fourth term).
\end{lem}
\begin{proof}
In fact if $R$ is any commutative ring and $I$ an ideal in $R$, then there is a natural exact sequence
\[
\xymatrix{
\sym{n+1}{R}{I}\ar[r]^-{\eta}
&\sym{n}{R}{I}\ar[r]
&\sym{n}{R/I}{I/I^2}\ar[r]
&0
}
\]
where $\eta(a_0\symm a_1\symm \cdots \symm a_n)=a_0\cdot(a_1\symm \cdots \symm a_n)=(a_0a_1)\symm \cdots \symm a_n$.

In the particular case $n=2$, $R=\igr{G}$, $I=\aug{G}$ this gives an exact sequence 
\[
\sym{3}{\igr{G}}{\aug{G}}\to\sym{2}{\igr{G}}{\aug{G}}\to\sym{2}{\Z}{{G}}\to 0
\]
since $\aug{G}/\aug{G}^2\cong G$.

Now there is a natural surjective homomorphism of graded $\igr{G}$-algebras
\begin{eqnarray*}
\sym{\bullet}{\igr{G}}{\aug{G}}&\to &\aug{G}^\bullet,\\
 a_1\symm\cdots \symm a_n&\mapsto& \pf{a_1}\cdots\pf{a_n}.
\end{eqnarray*}
This is an isomorphism in dimension $2$ (and,for trivial reasons, in dimensions $0$ and $1$).  To see this, apply the functor 
$-\otimes_{\igr{G}}\aug{G}$ to the short exact sequence 
\[
0\to \aug{G}\to\igr{G}\to \Z\to 0
\]
to obtain the exact sequence
\[
0\to \tor{\igr{G}}{1}{\Z}{\aug{G}}\to\tens{2}{\igr{G}}{\aug{G}}\to \aug{G}^2\to 0.
\]
But $\tor{\igr{G}}{1}{\Z}{\aug{G}}=\ho{1}{G}{\aug{G}}\cong \ho{2}{G}{\Z}\cong \Extpow{2}{\Z}{G}$. A straightforward calculation
now shows that the map 
\[
\Extpow{2}{\Z}{G}\cong\tor{\igr{G}}{1}{\Z}{\aug{G}}\to \tens{2}{\igr{G}}{\aug{G}}
\]
sends $g_1\wedge g_2$ to $\pf{g_1}\otimes\pf{g_2}-\pf{g_2}\otimes\pf{g_1}$. 

Finally, we observe that the image of the map
\[
\sym{3}{\igr{G}}{\aug{G}}\to\sym{2}{\igr{G}}{\aug{G}}\cong \aug{G}^2
\]
is clearly $\aug{G}^3$.

\end{proof}
\begin{rem}
It is a straightforward matter to verify that $\sym{\bullet}{\igr{G}}{\aug{G}}$ has the following presentation as a graded 
ring:  It is generated in degree $1$ by the elements $\pf{g}$, $g\in G$, subject to the relations
\begin{enumerate}
\item[(N)] $\pf{1}=0$
\item[(R)] $\pf{g_1}\symm\pf{g_2g_3}+\pf{g_2}\symm\pf{g_3}=\pf{g_1g_2}\symm\pf{g_3}+\pf{g_1}\symm\pf{g_2}$ for all $g_1,g_2,g_3\in G$. 
\item[(S)] $\pf{g_1}\symm\pf{g_2}=\pf{g_2}\symm\pf{g_1}$ for all $g_1,g_2\in G$
\end{enumerate}

For abelian groups the surjective homomorphism of graded rings 
$\alpha:\sym{\bullet}{\igr{G}}{\aug{G}}\to\aug{G}^\bullet$ is not generally injective in dimensions greater 
than $2$.  However, the following is known:  

If $G$ is either torsion-free or cyclic then $\alpha$ is an isomorphism (see Bak and Tang \cite{bak:tang}). 
On the other hand, if $G$ is an elementary 
abelian $2$-group, then the kernel of $\alpha$ is the ideal generated by the degree $3$ terms $\pf{g_1}\symm\pf{g_2}\symm\pf{g_1g_2}$ 
for $g_1,g_2\in G$  (see Bak and Vavilov \cite{bak:vavilov}). It is easy to see that these latter terms are nonzero 
(by considering their image in $\sym{3}{\Z/2}{G}$, for example). 
\end{rem}

Applying Lemma \ref{lem:sym} to the case $G=\sq{F}$ gives:

\begin{cor} \label{cor:sym}
Let $F$ be a field. There is a natural exact sequence of $\sgr{F}$-modules 
\[
0\to \aug{F}^3\to\aug{F}^2\to \sym{2}{\F{2}}{\sq{F}}\to 0.
\]
\end{cor}

On the other hand,  clearly there is also a natural  homomorphism of additive groups 
\[
\xymatrix{
\asym{2}{\Z}{F^\times}\ar@{>>}[r]
&\asym{2}{\Z}{F^\times}\otimes \Z/2\ar[r]^-{\cong}
&\sym{2}{\F{2}}{\sq{F}}.
}
\]

For any field $F$ we define the $\sgr{F}$-module
\begin{eqnarray*}
\rasym{2}{\Z}{F^\times}:=& \aug{F}^2\times_{\sym{2}{\F{2}}{\sq{F}}}\asym{2}{\Z}{F^\times}\subset \aug{F}^2\oplus \asym{2}{\Z}{F^\times}
\end{eqnarray*}
where $\asym{2}{\Z}{F^\times}$ has the trivial $\sgr{F}$-module structure.

Given $a,b\in F^\times$, we let $\rsf{a}{b}$ denote the element 
\[
\rsf{a}{b}:= (\pf{a}\pf{b},a\asymm b)\in \rasym{2}{\Z}{F^\times}. 
\]

\begin{lem} Let $F$ be a field.
\begin{enumerate}
\item $\aug{F}\rasym{2}{\Z}{F^\times}\cong\aug{F}^3$
\item $\rasym{2}{\Z}{F^\times}_{F^\times}\cong \asym{2}{\Z}{F^\times}$
\item $\rasym{2}{\Z}{F^\times}$ is generated as an $\sgr{F}$-module by the elements $\rsf{a}{b}$, $a,b\in F^\times$.
\end{enumerate}
\end{lem}

\begin{proof}
\begin{enumerate}
\item We have an injective homomorphism 
\[
 \aug{F}^3\to \aug{F}^2\oplus \asym{2}{\Z}{F^\times},\quad x\mapsto (x,0).
\]
 But, since $x$ maps to $0$ in 
$\sym{2}{\F{2}}{\sq{F}}=\aug{F}^2/\aug{F}^3$, in fact $\aug{F}^3\subset \rasym{2}{\Z}{F^\times}$. 

On the other hand, if $a,b,c\in F^\times$, then $\pf{a}\rsf{b}{c}=(\pf{a}\pf{b}\pf{c},0)$. 
So\\
 $\aug{F}^3\subset \aug{F}\rasym{2}{\Z}{F^\times}$.  

Conversely, if $(x,y)\in \rasym{2}{\Z}{F^\times}$ and $a\in F^\times$, then $\pf{a}(x,y)=(\pf{a}x,0)\in \aug{F}^3$; i.e.  
$\aug{F}\rasym{2}{\Z}{F^\times}\subset \aug{F}^3$.

\item Suppose that $(x,y)$ lies in the kernel of the surjective $\sgr{F}$-homomorphism $\rasym{2}{\Z}{F^\times}\to
\asym{2}{\Z}{F^\times}$. Then $y=0$ and thus $x$ maps to $0$ in $\sym{2}{\F{2}}{\sq{F}}$. By Corollary \ref{cor:sym}, 
$x\in\aug{F}^3=$ and hence $(x,y)\in \aug{F}\rasym{2}{\Z}{F^\times}$.

Observe that it follows that there is a natural short exact sequence of $\sgr{F}$-modules
\[
0\to \aug{F}^3\to \rasym{2}{\Z}{F^\times}\to\asym{2}{\Z}{F^\times}\to 0.
\]

\item Let $K(F)$ be the $\sgr{F}$-submodule of $\rasym{2}{\Z}{F^\times}$ generated by the elements 
$\rsf{a}{b}$. Since 
$\pf{a}\rsf{b}{c}=\pf{a}\pf{b}\pf{c}\in\aug{F}^3\subset \rasym{2}{\Z}{F^\times}$, it follows that $\aug{F}^3\subset K(F)$. 
On the other hand the homomorphism $\rasym{2}{\Z}{F^\times}\to \asym{2}{\Z}{F^\times}$ maps 
$K(F)$ onto $\asym{2}{\Z}{F^\times}$, since 
the latter is generated by the elements $a\asymm b$. Thus $K(F)=\rasym{2}{\Z}{F^\times}$ as required. 
\end{enumerate}
\end{proof}

Observe that the $\sgr{F}$-module structure on $\rasym{2}{\Z}{F^\times}$ is given by the formula
\[
\an{b}\rsf{a}{c}=\rsf{ab}{c}-\rsf{b}{c}=\rsf{a}{bc}-\rsf{a}{b}.
\]

We define the \emph{refined Bloch-Wigner homomorphism} $\bw$ to be the $\sgr{F}$-module homomorphism
\[
\bw :\rpb{F}\to\rasym{2}{\Z}{F},\qquad \gpb{x}\mapsto \rsf{1-x}{x}.
\]

In view of the definition of $\rasym{2}{\Z}{F^\times}$, we can express $\bw= (\lambda_1,\lambda_2)$ where 
$\lambda_1:\rpb{F}\to \aug{F}^2$ is the map $\gpb{x}\mapsto \pf{1-x}\pf{x}$, and $\lambda_2$ is the composite
\[
\xymatrix{
\rpb{F}\ar@{>>}[r]
&\pb{F}\ar[r]^-{\lambda}
&\asym{2}{\Z}{F^\times}.
}
\] 

It is a tedious calculation to verify directly $\lambda_1$ is a well-defined homomorphism of $\sgr{F}$-modules. However, we will see 
below that $\lambda_1$ arises naturally as a differential in a spectral sequence.

Recall that the homology groups $\ho{k}{\spl{2}{F}}{\Z}$ are naturally $\sgr{F}$-modules for all $k$. 

\begin{thm}\label{thm:mat}
For any field $F$ with at least $10$ elements, there is a natural surjective $\sgr{F}$-module homomorphism
\[
\rasym{2}{\Z}{F^\times}\to \ho{2}{\spl{2}{F}}{\Z}
\]
inducing an isomorphism $\coker{\bw}\cong \ho{2}{\spl{2}{F}}{\Z}$.
\end{thm}

\begin{proof}
Suppose first that $F$ is finite. Then $\ho{2}{\spl{2}{F}}{\Z}=0$. The statement of the theorem amounts to the surjectivity 
of $\bw$. 

For a finite field, since $F^\times$ is cyclic, $\asym{2}{\Z}{F^\times}=\sym{2}{\Z/2}{\sq{F}}$ and thus 
$\rasym{2}{\Z}{F^\times}=\aug{F}^2$. 

Recall that the Grothendieck-Witt ring of the field $F$ is the ring $\gw{F}=\sgr{F}/\gwrel{F}$, where $\gwrel{F}$ is the ideal generated by the 
elements $\pf{1-x}\pf{x}$. Thus $\coker{\bw}$ is $\gwaug{F}^2$, where $\gwaug{F}$ is the fundamental 
ideal $\aug{F}/\gwrel{F}$ in $\gw{F}$. 
However, it is well-known that $\gwaug{F}^2=0$ for any finite field $F$ (see, for example, \cite{milnor:huse}, section III.5).  

Thus we can suppose that $F$ is infinite.  In this case, the symplectic case of the theorem of 
Matsumoto and Moore (\cite{mat:pres}), gives a presentation of the 
group $\ho{2}{\spl{2}{F}}{\Z}$. It has the following form: 
The generators are symbols $\an{a_1,a_2}$, $a_i\in F^\times$, subject to the relations:
\begin{enumerate}
\item[(i)] $\an{a_1,a_2}=0$ if $a_i=1$ for some $i$
\item[(ii)] $\an{a_1,a_2}=\an{a_2^{-1}, a_1}$
\item[(iii)] $\an{a_{1},a_2a'_2} +\an{a_2,a'_2}=\an{a_{1}a_2,a'_2}+\an{a_1,a_2}$
\item[(iv)] $\an{a_1,a_2}=\an{a_1,-a_{1}a_2}$
\item[(v)] $\an{a_1,a_2}=\an{a_{1},(1-a_{1})a_2}$       
\end{enumerate}

Furthermore, Suslin has shown (\cite{sus:tors}, appendix) that for an infinite field $F$, there is an 
isomorphism of $\sgr{F}$-modules 
\[
\ho{2}{\spl{2}{F}}{\Z}\cong \gwaug{F}^2\times_{\gwaug{F}^2/\gwaug{F}^3}\milk{2}{F}, 
\quad \an{a,b}\leftrightarrow (\pf{a}\pf{b},\{ a,b\}).
\]

Now we have a map of diagrams of $\sgr{F}$-modules 
\[
\xymatrix{
&
\asym{2}{\Z}{F^\times}\ar[d]
&&&&
\milk{2}{F}\ar[d]\\
\aug{F}^2\ar[r]
&\sym{2}{\Z}{\sq{F}}
&
\ar@<4ex>[r]
&&
\gwaug{F}^2\ar[r]
&
\gwaug{F}^2/\gwaug{F}^3
}
\]
which induces a map of pullbacks
\[
\rasym{2}{\Z}{F^\times}\to \gwaug{F}^2\times_{\gwaug{F}^2/\gwaug{F}^3}\milk{2}{F}\cong \ho{2}{\spl{2}{F}}{\Z}
\]
sending the elements $\rsf{a}{b}$ to the elements $\an{a,b}$. 

This map is surjective, since the elements $\an{a,b}$ generate $\ho{2}{\spl{2}{F}}{\Z}$, and the image of $\bw$ is 
contained in its kernel 
since $\{1-x,x\}=0$ in $\milk{2}{F}$ and $\pf{1-x}\pf{x}=0$ in $\gwaug{F}^2$.  

To complete the proof of the theorem we must show that there is an $\sgr{F}$-homomorphism $\ho{2}{\spl{2}{F}}{\Z}\to \coker{\bw}$ sending 
$\an{a,b}$ to $\rsf{a}{b}\pmod{\image{\bw}}$; i.e. we must show that the elements $\rsf{a}{b}\in \rasym{2}{\Z}{F^\times}$ satisfy 
the Matsumoto-Moore relations modulo the image of $\bw$.

Now the elements $\rsf{a}{b}$ are easily seen to satisfy relations (i), (ii) and (iii). On the other hand, since 
$\rsf{a}{1-a}\equiv 0 \pmod{\image{\bw}}$ for all $a\in F^\times$, and since $\bw$ is an $\sgr{F}$-homomorphism, it follows that 
$0\equiv\an{b}\rsf{a}{1-a}\equiv\rsf{a}{(1-a)b}-\rsf{a}{b}\pmod{\image{\bw}}$ for all $a,b\in F^\times$. 

Now, for any $a\in F^\times$, $\bw(\gpb{a}+\an{-1}\gpb{a^{-1}})=\rsf{-a}{a}$, since 
\begin{eqnarray*}
\lambda_1(\gpb{a}+\an{-1}\gpb{a^{-1}})&=&\pf{1-a}\pf{a}+\an{-1}\pf{a(a-1)}\pf{a}\\
&=&\left(\pf{(1-a)a}-\pf{1-a}-\pf{a}\right)+\an{-1}\left(\pf{a-1}-\pf{a(a-1)}-\pf{a}\right)\\
&=&\pf{(1-a)a}-\pf{1-a}-\pf{a}+\pf{1-a}-\pf{a(1-a)}-\pf{-a}+\pf{-1}\\
&=&\pf{-1}-\pf{a}-\pf{-a}\\
&=&\pf{a}\pf{-a}
\end{eqnarray*}   
and
\begin{eqnarray*}
\lambda_2(\gpb{a}+\an{-1}\gpb{a^{-1}})&=&(1-a)\asymm a +(1-a^{-1})\asymm a^{-1}\\
=(1-a)\asymm a -\left(\frac{1-a}{-a}\right)\asymm a&=& (-a)\asymm a.
\end{eqnarray*}

Thus $\rsf{a}{-a}\equiv 0 \pmod{\image{\bw}}$ for all $a\in F^\times$ and hence 
$0\equiv \an{b}\rsf{a}{-a}\equiv \rsf{a}{-ab}-\rsf{a}{b}\pmod{\image{\bw}}$ for all $a,b\in F^\times$. Thus relations (iv) and (v) 
also hold in $\coker{\bw}$, and the theorem is proven.
\end{proof}
\begin{rem} The restriction to fields with at least $10$ elements is to rule out the exceptional cases of the field 
with $4$ and the field with $9$ elements for which $\ho{2}{\spl{2}{F}}{\Z}=\Z/p$ is nonzero (see 
Lemma \ref{lem:charp} below)
\end{rem}
Finally, we can define the 
\emph{refined Bloch group} of the field $F$ (with at least $4$ elements) to be the $\sgr{F}$-module 
\[
\rbl{F}:=\ker{\bw: \rpb{F}\to \rasym{2}{\Z}{F^\times}}.
\]

Thus we have:
\begin{cor}
For any field $F$ (with at least $10$ elements) there is an exact sequence of $\sgr{F}$-modules 
\[
0\to\rbl{F}\to \rpb{F}\to \rasym{2}{\Z}{F^\times}\to\ho{2}{\spl{2}{F}}{\Z}\to 0.
\]
\end{cor}

For future reference, we make the following observation:

\begin{lem}\label{lem:rblfin}
Let $F$ be a finite field with at least $4$ elements. Then the natural map $\rbl{F}\to \bl{F}$ induces an isomorphism
$\rbl{F}_{F^\times}\cong \bl{F}$. 
\end{lem}

\begin{proof}
We can assume $F$ has odd characteristic, since otherwise $\sq{F}=\{ 1\}$ and $\rpb{F}=\pb{F}$, $\rbl{F}=\bl{F}$. 

Thus if $a$ is a generator of $F^\times$, $G:=\sq{F}$ is cyclic of order $2$ generated by the class $\an{a}$. Thus,
$\aug{F}=\Z\cdot\pf{a}$ is infinite cyclic, and $\aug{F}^n=2^{n-1}\aug{F}$ for all $n\geq 1$ (since $\pf{a}^2=-2\pf{a}$). 

Since $F^\times$ is cyclic,  $\asym{2}{\Z}{F^\times}=\sym{2}{\Z}{\sq{F}}$ and 
thus $\rasym{2}{\Z}{F^\times}=\aug{F}^2$. The fact 
that $\gwaug{F}^2=0$ thus amounts to the statement that the map $\bw:\rpb{F}\to \aug{F}^2$ is surjective.

Thus taking $G$-coinvariants of the short exact sequence
\[
0\to \rbl{F}\to\rpb{F}\to\aug{F}^2\to 0
\] 
gives the exact sequence
\[
\xymatrix{
\ho{1}{G}{\aug{F}^2}\ar[r]
&\rbl{F}_{F^\times}\ar[r]
&\pb{F}\ar[r]^-{\lambda}
&\sym{2}{\Z}{\sq{F}}\ar[r]
&0.
}
\]
However, $\aug{F}^2\cong \Z$ and the generator $\an{a}$ of the cyclic group 
$G$ acts as $-1$ (since $\an{a}\pf{a}=-\pf{a}$). 
Thus $\ho{1}{G}{\aug{F}^2}=0$ and $\rbl{F}_{F^\times}=\bl{F}$ as required.
\end{proof}

\begin{rem}
We will show below that for a finite field $F$ the action of $F^\times$ on $\rbl{F}$ is trivial. 
It will thus follow that 
$\rbl{F}=\bl{F}$ when $F$ is finite. 

In general the action of $F^\times$ on $\rbl{F}$ is nontrivial (for example, if $F$ is a local or global field). 
However, we will see 
below that for any field $F$ 
the natural map $\rbl{F}\to\bl{F}$ is always surjective 
and that the induced surjective map $\rbl{F}_{F^\times}\to\bl{F}$ has a kernel 
annihilated by a power of $2$. 
\end{rem}

\section{The   homology of $\mathrm{SL_2}$ of finite fields} \label{sec:finite}

In this section $p$ is a prime number, $q=p^f$ for some $f\geq 1$, and $\F{q}$ denotes the finite field with 
$q$ elements.
Recall that the group $\spl{2}{\F{q}}$ has order $q(q^2-1)=q(q-1)(q+1)$. We review - for want of an explicit 
reference - some of the main facts about the integral homomlogy of $\spl{2}{\F{q}}$. 

We recall the relevant results we will use (for details, see \cite{brown:coh}, III.9, III.10):  
Let $G$ be a finite group, $\ell$ a prime number and $H$ a Sylow $\ell$-subgroup of $G$. 
For any $g\in G$, conjugation by $g$ 
induces a homomorphism $\ho{k}{H}{\Z}\to \ho{k}{gHg^{-1}}{\Z}, z\mapsto g\cdot z$. 
For $g\in G$, we say that $z\in\ho{k}{H}{\Z}$ is \emph{$g$-invariant} if 
$\rs^H_{H\cap gHg^{-1}}z
=\rs^{gHg^{-1}}_{H\cap gHg^{-1}}g\cdot z$.

Let
\[
\hinv{G}{\ho{k}{H}{\Z}}:= \{ z\in \ho{k}{H}{\Z}\ |\  z \mbox{ is $g$-invariant for all }
g\in G\}.
\]
Then for $k\geq 1$, the corestriction homomorphism 
\[
\cores^G_H:\ho{k}{H}{\Z}\to \ho{k}{G}{\Z}
\]
induces an isomorphism
\[
\hinv{G}{\ho{k}{H}{\Z}}\cong\psyl{\ho{k}{G}{\Z}}{\ell}=\ho{k}{G}{\psyl{\Z}{\ell}}.
\]

Now, for $\ell\not= p$, the $\ell$-Sylow subgroups of $G=\spl{2}{\F{q}}$ are cyclic or generalised quaternion, and hence the (co)homology
is $\ell$-periodic. This means that there is a number $d=d(\ell)\geq 2$ such that  $\ho{k}{G}{\psyl{\Z}{\ell}}\cong \ho{k+d}{G}{
\psyl{\Z}{\ell}}$ for all $k\geq 1$ and that this happens if and only if $\ho{d-1}{G}{\psyl{\Z}{\ell}}\cong \psyl{\Z}{\ell}/|G|$. In 
this case, $d=d(\ell)$ is called the $\ell$-period of $G$.

We recall also the following useful results of Swan \cite{swan:period}: 
\begin{thm}\label{thm:swan12}[Swan \cite{swan:period}, Theorems 1 and 2]
\begin{enumerate}
\item Suppose that  $\ell$ is odd and the the $\ell$-Sylow subgroup of $G$ is cyclic. Let $H$ be a $\ell$-Sylow subgroup of 
$G$ and let $\Phi_\ell$ be the group of 
automorphisms of $H$ induced by inner automorphisms of $G$. Then the $\ell$-period of $G$ 
is $2\cdot|\Phi_\ell|$.
\item If the $2$-Sylow subgroup of $G$ is cyclic, the $2$-period is $2$. If the $2$-Sylow subgroup of $G$ is generalised quaternion then 
the $2$-period is $4$.
\end{enumerate}
\end{thm}

Furthermore, we recall the well-known calculation
\begin{lem}\label{lem:h1}
\[
\ho{1}{\spl{2}{\F{q}}}{\Z}= 
\left\{
\begin{array}{ll}
\Z/p,& q=2,3\\
0,&\mbox{otherwise}
\end{array}
\right.
\]
\end{lem}
\begin{cor} \label{cor:2}
Suppose $p$ is odd. Then the $2$-period of $\spl{2}{\F{q}}$ is $4$ and for $k\geq 1$
\[
\ho{k}{\spl{2}{\F{q}}}{\psyl{\Z}{2}}=
\left\{
\begin{array}{ll}
\psyl{\Z}{2}/q(q^2-1),& k\equiv 3\pmod{4}\\
0,&\mbox{otherwise}
\end{array}
\right.
\] 
\end{cor}
\begin{proof}
When $p\not= 2$, the $2$-Sylow subgroups of $\spl{2}{\F{q}}$ are generalized quaternion groups. 
So Swan's theorem tells us that the  $2$-period is $4$, and hence that $\ho{k}{\spl{2}{\F{q}}}{\psyl{\Z}{2}}=\psyl{\Z}{2}/q(q^2-1)$ 
whenever $k\equiv 3\pmod{3}$. On the other hand, the even-dimensional integral homology of the  (generalized) quaternion groups are zero.

Finally, by Lemma \ref{lem:h1}, $\ho{k}{\spl{2}{\F{q}}}{\psyl{\Z}{2}}=0$ for $k\equiv 1\pmod{4}$.  
\end{proof}

Now we consider the $\ell$-Sylow subgroups for odd $\ell$ dividing $q-1$.
We let $T$ denote the diagonal subgroup $\{ D(x)\ |\ x\in F^\times\}$ of $\spl{2}{F}$. 
$T$ is cyclic of order $q-1$. Since the order of $\spl{2}{F}$ is $q(q^2-1)$, 
it follows that for any \emph{odd} prime $\ell$ dividing $q-1$, $\psyl{T}{\ell}$ is 
a Sylow $\ell$-subgroup of $\spl{2}{F}$.

\begin{lem}\label{lem:q-1} Let $\ell$ be an odd prime dividing $q-1$. The $\ell$-period of $\spl{2}{\F{q}}$ is $4$ and for $k\geq 1$
\[
\ho{k}{\spl{2}{\F{q}}}{\psyl{\Z}{\ell}}=
\left\{
\begin{array}{ll}
\psyl{\Z}{\ell}/q(q^2-1),& k\equiv 3\pmod{4}\\
0,&\mbox{otherwise}
\end{array}
\right.
\]
\end{lem}
\begin{proof}
Fix $\ell$ odd dividing $q-1$, and let $D(a)\in \psyl{T}{\ell}$. So $a^2\not=1$. 

Now  suppose that 
\[
A=\matr{x}{y}{z}{w}
\]
normalises $\psyl{T}{\ell}$.

Then since
\[
\matr{x}{y}{z}{w}\matr{a}{0}{0}{a^{-1}}\matr{w}{-y}{-z}{x}=
\matr{axw-a^{-1}yz}{(a^{-1}-a)xy}{(a-a^{-1})zw}{a^{-1}xw-ayz}\in T
\]
it follows that $xy=zw=0$. Thus $A$ belongs to the group of order $2(q-1)$ 
generated by $T$ and 
\[
w=\matr{0}{1}{-1}{0}.
\]
Conjugation by $w$ acts as $-1$ (i.e inversion) on $\psyl{T}{\ell}$. Hence $\Phi_{\ell}=\an{-1}$ (in the notation of 
Theorem \ref{thm:swan12} above). So the $\ell$-period is $2|\Phi_\ell|=4$. The rest follows as in the case $\ell=2$.
\end{proof}

Next, we deal with the case $\ell | q+1$ ($\ell$ odd).  Let $E/\F{q}$ be a quadratic extension of fields.(So $E=\F{q^2}$.) 
Then there is an embedding of groups $E^\times \to \Aut{\F{q}}{E}, 
a\mapsto \mu_a$, and $\Aut{\F{q}}{E}\cong\gl{2}{\F{q}}$ on choosing  an $\F{q}$-basis of $E$. 
The composite 
\[
\xymatrix{E^\times\ar[r]& \Aut{\F{q}}{E}\ar[r]^-{\det}&\F{q}^\times}
\]
is just the norm map $N_{E/\F{q}}$. Thus, if we let  $K=\ker{N_{E/\F{q}}:E^\times\to F^\times}$, 
we obtain an embedding $K\to \spl{2}{\F{q}}$ on choosing an $\F{q}$-basis of $E$. Since the 
norm map is surjective, it follows that $K$ is cyclic of order $q+1$. Thus for any odd 
$\ell$ dividing $q+1$, $\psyl{K}{\ell}$ is a Sylow $\ell$-subgroup of $\spl{2}{\F{q}}$. 

\begin{lem}\label{lem:q+1} 
Let $\ell$ be an odd prime dividing $q+1$. The $\ell$-period of $\spl{2}{\F{q}}$ is $4$ and for $k\geq 1$
\[
\ho{k}{\spl{2}{\F{q}}}{\psyl{\Z}{\ell}}=
\left\{
\begin{array}{ll}
\psyl{\Z}{\ell}/q(q^2-1),& k\equiv 3\pmod{4}\\
0,&\mbox{otherwise}
\end{array}
\right.
\]
\end{lem}
\begin{proof} Let $\mu:E^\times\to \gl{2}{\F{q}}$ be the embedding described above. 
If  $\sigma\in\gal{E}{\F{q}}$ then $\sigma\in\Aut{\F{q}}{E}$ and thus, given the choice of basis, 
is represented by an element $\tilde{\sigma}\in\gl{2}{\F{q}}$. 
Then $\Gamma:=\mu(E^\times)\cdot\< \tilde{\sigma}\>\subset \gl{2}{\F{q}}$ is a semidirect 
product in which the $\tilde{\sigma}$  by conjugation  on $\mu(E^\times)$ corresponds to 
the Galois action of $\sigma$ on $E^\times$.  

Fix an odd prime $\ell$ dividing $q+1$. So $\mu(\psyl{K}{\ell})$ is an $\ell$-Sylow subgroup of $\\spl{2}{\F{q}}$.
We will show that the normalizer $\mu(\psyl{K}{\ell})$ in $\spl{2}{\F{q}}$ is  $\Gamma\cap\spl{2}{F}$. 

Since the elements of $\mu(E^\times)$ act trivially by conjugation on $\mu(K)$ and since 
$\tilde{\sigma}$ acts by inversion on $\mu(K)$ the result follows from this as in the case $\ell|q-1$.

We must distinguish two cases:
\begin{enumerate}
\item[Case (i):] $p\not= 2$.

Let $E=\F{q}(\theta)$ where $\theta^2=\alpha\in \F{q}$ (and thus $\alpha$ is not 
square in $\F{q}$). Take the basis $\{ 1,\theta\}$ of $E$. Then the associated embedding 
$\mu: K\to \spl{2}{\F{q}}$ is given by
\[
a+b\theta\mapsto \matr{a}{b\alpha}{b}{a}. 
\]

  Suppose that 
\[
A=\matr{x}{y}{z}{w}\in \spl{2}{\F{q}}
\]
normalizes $\mu(\psyl{K}{\ell})$. 

Let $\lambda=a+b\theta\in \psyl{K}{\ell}$. Since $\ell$ does not divide $q-1$,
$\lambda\not\in \F{q}$ (so that $b\not=0$). Then from
\[
A\mu(\lambda)A^{-1}=
\matr{a+b(yw-xz\alpha)}{b(x^2\alpha-y^2)}{b(w^2-z^2\alpha)}{a-b(yw-xz\alpha)}
\in\mu(K)
\]
we obtain the conditions
\begin{eqnarray*}
yw-xz\alpha&=&0\\
x^2\alpha-y^2&=&(w^2-z^2\alpha)\alpha\\
\end{eqnarray*}
(since $2b\not= 0$). 

Now we fix $x$ and $z$. Eliminating $w$ from these equations gives the quartic
\[
(y^2-x^2\alpha)(y^2-z^2\alpha^2)=0
\]
in $y$. Since $\alpha$ is not square in $F$, this leads to the two solutions 
$(y,w)=(z\alpha,x)$ and $(-z\alpha,-x)$. The first of these gives 
\[
A=\matr{x}{z\alpha}{z}{x}\in \mu(E^\times)
\] 
and the second gives 
\[
A=\matr{x}{-z\alpha}{z}{-x}=\matr{x}{z\alpha}{z}{x}\matr{1}{0}{0}{-1}=
\matr{x}{z\alpha}{z}{x}\tilde{\sigma}\in \Gamma.
\]

So the normalizer of $\mu(\psyl{K}{\ell})$ is contained in $\Gamma$ as required.

\item[Case (ii):] $p=2$

We write $E=\F{q}(\theta)$ where $\theta$ satisfies $\theta^2+\theta=\alpha$ and 
$\sigma(\theta)=1+\theta$. Again we choose the basis $\{ 1,\theta\}$ of $E$. 
The embedding $\mu:E^\times\to \gl{2}{\F{q}}$ then has the form
\[
a+b\theta\mapsto \matr{a}{b\alpha}{b}{a+b}
\]
and we have 
\[
\tilde{\sigma}=\matr{1}{1}{0}{1}.
\]

Suppose again that 
\[
A=\matr{x}{y}{z}{w}\in \spl{2}{F}
\] 
normalizes $\mu(\psyl{K}{\ell})$ 
and that $\lambda=a+b\theta\in \psyl{K}{\ell}$ (so that $b\not= 0$ as above).

Then from 
\[
A\mu(\lambda)A^{-1}=
\matr{a(xw+yz)+b(xz\alpha+yz+yw)}{b(x^2\alpha+y^2+xy)}{b(w^2+z^2\alpha+zw)}
{a(xw+zy)+b(xz\alpha+xw+yw)}\in \mu(E^\times)
\]
we get the conditions:
\begin{eqnarray*}
xw+yz&=&z^2\alpha+w^2+zw\\
x^2\alpha+xy+y^2&=&(z^2\alpha+w^2+zw)\alpha.
\end{eqnarray*}
If we fix $x$ and $z$, then the four solutions of this pair of binary quadratic equations 
is $(y,w)=(z\alpha, 0), (z\alpha, x+z), (x+z\alpha, x)$ and  $(x+z\alpha,z)$.

The first and last of these give singular matrices.  The second gives 
\[
A=\matr{x}{z\alpha}{z}{x+z}=\mu(x+z\theta)\in \mu(E^\times).
\]

The third solution gives 
\[
A=\matr{x}{x+z\alpha}{z}{x}=\matr{x}{z\alpha}{z}{x+z}\cdot\tilde{\sigma}\in \Gamma.
\]

Thus, again, the normalizer of $\mu(\psyl{K}{\ell})$ is contained in $\Gamma$ and the lemma is proven.
\end{enumerate}
\end{proof}

Pulling together the statements of Corollary \ref{cor:2} and Lemmas \ref{lem:q-1} and \ref{lem:q+1} we have 
\begin{cor}\label{cor:1/p}
For all $k\geq 1$ 
\[
\ho{k}{\spl{2}{\F{q}}}{\Z[1/p]}=
\left\{
\begin{array}{ll}
\Z/(q^2-1),& k\equiv 3\pmod{4}\\
0,&\mbox{otherwise}
\end{array}
\right.
\]

\end{cor}

Furthermore, we can deduce the following:
\begin{cor} \label{cor:subgp}
Let $H$ be any subgroup of $\spl{2}{\F{q}}$ of order not divisible by $p$. If $k\equiv 3\pmod{4}$ then 
$\ho{k}{H}{\Z}=\Z/|H|$ and the 
corestriction map $\ho{k}{H}{\Z}\to\ho{k}{\spl{2}{\F{q}}}{\Z}$ is injective. 
\end{cor}
\begin{proof}
It is clearly sufficient to consider the case where $H$ is an $\ell$-group for some prime $\ell\not= p$. In particular, $H$ is cyclic 
or generalised quaternion and $\ho{k}{H}{\Z}=\Z/|H|$. Now $H$ is contained in an $\ell$-Sylow subgroup, $L$ say, of $\spl{2}{\F{q}}$. 
It is a straightforward calculation to show that whenever $L$ is a cyclic $\ell$-group or generalized quaternion $2$-group 
and whenever $H$ is a subgroup that 
the corestriction map $\ho{k}{H}{\Z}\to\ho{k}{L}{\Z}$ is injective (for $k\equiv 3\pmod{4}$).
 On the other hand, the results above show that the corestriction 
map $\ho{k}{L}{\Z}\to\ho{k}{\spl{2}{F}}{\Z}$ is injective.  
\end{proof}

We will also need  the following result in the next section:

\begin{lem}\label{lem:trivial} 
For all $k$, the natural action of $\F{q}^\times$ on $\ho{k}{\spl{2}{\F{q}}}{\Z[1/p]}$ is trivial.
\end{lem}

\begin{proof} By Corollary \ref{cor:1/p}, we can assume $k\equiv 3\pmod{4}$.
 
By Corollary \ref{cor:subgp}, the corestriction map 
 \[
\ho{k}{\spl{2}{\F{q}}}{\Z[1/p]}\to\ho{k}{\spl{2}{\F{q^2}}}{\Z[1/p]}
\]
is injective. But for any $a\in\F{q}^\times$, $a\in (\F{q^2}^\times)^2$ and thus $\an{a}=1$ in $\sq{\F{q^2}}$, so that $\an{a}$ 
acts trivially on $\ho{k}{\spl{2}{\F{q^2}}}{\Z[1/p]}$.
\end{proof}

\begin{cor}\label{cor:k3}
For any finite field $\F{q}$ there is a natural isomorphism
\[
\ho{3}{\spl{2}{\F{q}}}{\Z[1/p]}\cong \kind{\F{q}}=K_3(\F{q}).
\]
\end{cor}

\begin{proof}
For any field $F$, the Hurewicz map induces an isomorphism
\[
K_3(F)/({-1}\cdot K_2(F))\cong \ho{3}{\spl{}{F}}{\Z}
\]
(see Suslin, \cite{sus:bloch}, Corollary 5.2). Thus, for any field $F$ there is a natural composite homomorphism
\[
\ho{3}{\spl{2}{F}}{\Z}\to\ho{3}{\spl{}{F}}{\Z}\to\kind{F}.
\]
When $F$ is algebraically closed, this map is an isomorphism (Sah, \cite{sah:discrete3}). 

For a finite field $\F{q}$ we have 
\[
K_3(\F{q})=\kind{\F{q}}=\Z/(q^2-1)=\kind{\F{q}}\otimes\Z[1/p]
\]
and the functorial maps $K_3(\F{q})\to K_3(\F{q^n})$ are injective (Quillen, \cite{quillen:cohkfin})

Thus, if we let $\bar{\F{q}}$ denote an algebraic closure of $\F{q}$ we have a commutative diagram
\[
\xymatrix{
\ho{3}{\spl{2}{\F{q}}}{\Z[1/p]}\ar[r]\ar@{^{(}->}[d]&\kind{\F{q}}\ar@{^{(}->}[d]\\
\ho{3}{\spl{}{\bar{\F{q}}}}{\Z[1/p]}\ar[r]^-{\cong}&\kind{\bar{\F{q}}}
}
\]
in which the vertical arrows are injective, 
from which it follows that the top horizontal map is injective, and hence an isomorphism of finite abelian groups of equal order.
\end{proof}

In the remainder of this section, we will calculate, for completeness,  the $p$-Sylow subgroups of 
$\ho{k}{\spl{2}{\F{q}}}{\Z}$ for $k\leq 3$. 

Of course, for $g\in G$ and $z\in \ho{k}{H}{\Z}$ the condition 
\[
\rs^H_{H\cap gHg^{-1}}z
=\rs^{gHg^{-1}}_{H\cap gHg^{-1}}g\cdot z
\]
is trivially satisfied if $H\cap gHg^{-1}=\{ 1\}$. Thus, in order to determine
$\hinv{G}{\ho{k}{H}{\Z}}$ for an $\ell$-Sylow subgroup $H$, 
it is enough to consider only the set $\conj{G}{H}$ of those elements $g$ for which 
 $H\cap gHg^{-1}\not=\{ 1\}$.

A Sylow $p$-subgroup of $\spl{2}{\F{q}}$ is the group of unipotents
\[
U:= \left\{\matr{1}{a}{0}{1}\ |\ a\in \F{q}\right\}\cong \F{q}.
\]

We first determine the set $\conj{\spl{2}{\F{q}}}{U}$ of those $A\in \spl{2}{\F{q}}$ for which $AUA^{-1}\cap U\not= \{ 1\}$.  Let 
\[
A=\matr{x}{y}{z}{w}\in \spl{2}{\F{q}}.
\]
Then 
\[
A\matr{1}{a}{0}{1}A^{-1}=\matr{1-axz}{ax^2}{-az^2}{1+axz}.
\]
Thus $A\in\conj{\spl{2}{\F{q}}}{U}$ if and only if $z=0$ and in this case 
\[
A=\matr{x}{y}{0}{x^{-1}}\in B
\]
where $B$ is the subgroup of upper triangular matrices in $\spl{2}{\F{q}}$.

\begin{cor} \label{cor:B}
\begin{enumerate}
\item For all $k\geq 1$, we have 
\[
\ho{k}{B}{\Z}\cong\ho{k}{T}{\Z}\oplus \psyl{\ho{k}{B}{\Z}}{p}.
\]

\item The inclusion $B\to\spl{2}{\F{q}}$ induces an isomorphism
\[
\psyl{\ho{k}{B}{\Z}}{p}\cong\psyl{\ho{k}{\spl{2}{\F{q}}}{\Z}}{p}\cong \Inv{\ho{k}{\F{q}}{\Z}}{(\F{q}^\times)^2}.
\]
\end{enumerate}
\end{cor}

\begin{proof}

\begin{enumerate}
\item  The result follows from the Hochschild-Serre spectral sequence associated 
to the split short exact sequence 
\[
1\to U\to B\to T\to1
\]
together with the fact that $(|T|,|B|)=1$.

\item The Sylow $p$-subgroup $U$ of $B$ is also a Sylow $p$-subgroup 
of $\spl{2}{\F{q}}$ and the calculations above show that $\conj{\spl{2}{\F{q}}}{U}=\conj{B}{U}=B$.
Thus
\[
\psyl{\ho{k}{B}{\Z}}{p}=\hinv{B}{\ho{k}{U}{\Z}}=\hinv{\spl{2}{F}}{\ho{k}{U}{\Z}}=
\psyl{\ho{k}{\spl{2}{F}}{\Z}}{p}.
\]
The final isomorphism derives from the fact that $U\cong \F{q}$ is normal in $B$ with quotient $T\cong \F{q}^\times$ and with 
these identifications $a\in \F{q}^\times$ acts by conjugation on $U=\F{q}$ as multiplication by $a^2$.

\end{enumerate}
\end{proof}
 
\begin{lem}\label{lem:p} $\spl{2}{\F{p}}$ is $p$ periodic with $p$-period $d=d(p)$ given by 
\[
d(p)=\left\{
\begin{array}{ll}
2,&p=2\\
p-1,&p\not= 2
\end{array}
\right.
\]

Furthermore, for $k\geq 1$
\[
\psyl{\ho{k}{\spl{2}{\F{p}}}{\Z}}{p}=
\left\{
\begin{array}{ll}
\Z/p, & k\equiv -1\pmod{d(p)}\\
0, &\mbox{otherwise}
\end{array}
\right.
\]
\end{lem}
\begin{proof}
The first statement follows from Swan's Theorem, since $\Phi_p\cong (\F{p}^\times)^2$.

For the second statement, we can suppose $p$ is odd and let $x\in (\F{p}^\times)^2$ of order $(p-1)/2$. Then multiplication by 
$x$ on $\F{p}=\Z/p$ induces multiplication by $x^{k+1}$ on $\ho{2k+1}{\F{p}}{\Z}=\Z/p$. The statement now follows easily, since $\Z/p$ 
is invariant only if $x^{k+1}=1$. 
\end{proof}

When $q=p^f>p$, of course the integral homology is no longer $p$-periodic. However, we will calculate 
$\psyl{\ho{k}{\spl{2}{\F{q}}}{\Z}}{p}$ for $k\leq 3$.

The following is a minor variation on \cite{sus:homgln}, Lemma 1.:

\begin{lem}\label{lem:sus}
Let $m,n\geq 1$. Suppose that 
$(p-1)f>mn$. Then there exists $a\in \F{q}^\times$ such that 
for any (not necessarily distinct) $\phi_1,\ldots,\phi_n\in \gal{\F{q}}{\F{p}}$ we have $\prod_{i=1}^n\phi(a^m)\not=1$.
\end{lem}
\begin{proof}
Let $a$ be a generator of $\F{q}^\times$. Suppose, for the sake of contradiction that there 
exist $\phi_1,\ldots,\phi_n\in\gal{\F{q}}{\F{p}}$ with $\prod_i\phi_i(a^m)=1$. Since $\gal{\F{q}}{\F{p}}$ 
is generated by $x\mapsto  x^p$, for each $i\in\{1,\dots,n\}$ there exists $k_i<f$ such that $\phi_i(x)=x^{p^{k_i}}$. Hence
\[
a^{\sum_{i=1}^nmp^{k_i}}=1
\]
and thus $\sum_{i=1}^nmp^{k_i}\equiv 0\pmod{p^f-1}$. For $0\leq t\leq f-1$ let 
$s_t\geq 0$ be the number of $i$ for which $k_i=t$. Thus
\[
\sum_{t=0}^{f-1}ms_tp^t\equiv 0\pmod{p^f-1}.
\]

Let $k_t=ms_t$ for $t<f$. Then $\sum k_t\geq (p-1)f$. For if some $k_t\geq p$, then 
replacing $k_t$ by $k_t-p$ and (ordering the $t$ cyclically) $k_{t+1}$ by $k_{t+1}+1$ 
we get a new system of $k_t$ satisfying the same congruence but having a smaller sum. By iterating this operation we arrive at a collection $k_1',\ldots,k_{f-1}'$ satisfying the congruence and also $k_t'<p$ for all $t$ and $\sum k_t'\leq \sum k_t$. But then the 
inequalities 
\[
p^f-1\leq\sum_{t=1}^{f-1}k_t'p^t\leq (p-1)(1+p+\cdots+p^{f-1})=p^f-1
\]
imply that $k_t'=p-1$ for all $t$ and hence $\sum k_t\geq \sum k_t'=(p-1)f$. 

But this gives the contradiction $(p-1)f\leq \sum k_t =m\sum s_t=mn$, proving the lemma.
\end{proof}

\begin{cor} \label{cor:sus}
Suppose that  
$(p-1)f>mn$. Let $\F{q}^\times$ act on $\F{q}$ by multiplication. Then  
\[
\Inv{\tens{n}{\Z}{\F{q}}}{(\F{q}^\times)^m}=\Inv{\Extpow{n}{\Z}{\F{q}}}{(\F{q}^\times)^m}=0.
\] 
\end{cor}

\begin{proof}
By Lemma \ref{lem:sus}, there exists $a\in \F{q}^\times$ such that
if $b=a^m$ then $\prod_{i=1}^n\phi_i(b)\not=1$ for all $\phi_1,\ldots,\phi_n\in\gal{\F{q}}{\F{p}}$. 

Let 
$\mu_b$ denote the $\F{p}$-linear endomorphism $x\mapsto bx$ of $\F{q}$. Let $\phi$ 
be the $\F{p}$-linear endomorphism $\mu_b\otimes\cdots\otimes\mu_b$ of 
$\tens{n}{\F{p}}{\F{q}}$. Then $b$ acts as $\phi$ on $\tens{n}{\F{p}}{F}$. Thus
$b$
 fixes a nonzero element of this space if and only if $1$ is not an eigenvalue of 
$\phi$. The eigenvalues of $\phi$ are products of the form $\lambda_1\cdots\lambda_n$ 
where $\lambda_1,\ldots,\lambda_n$ are (not necessarily distinct) eigenvalues of 
$\mu_b$. But the eigenvalues of $\mu_b$ are precisely the values of  the 
elements of $\gal{\F{q}}{\F{p}}$ at $b$. This proves the result. 

The same argument applies to $\Extpow{n}{\Z}{\F{q}}=\Extpow{n}{\F{p}}{\F{q}}$, 
since it is a quotient module of $\tens{n}{\F{p}}{\F{q}}$ for the 
action of $\F{q}^\times$.  
\end{proof}

\begin{lem}\label{lem:charp}
$\psyl{\ho{k}{\spl{2}{\F{q}}}{\Z}}{p}=\Z/p$ if  
\[
(k,q)\in \{ (1,2),(1,3),(2,4),(2,9), (3,2),(3,3),(3,4),(3,5),(3,8),(3,9),(3,27)\}.
\]
and $\psyl{\ho{k}{\spl{2}{\F{q}}}{\Z}}{p}=0$ for any other value of $(k,q)$ with $k\leq 3$.
\end{lem}
\begin{proof}
The cases $k=1$ or $q=p$ are already covered by Lemmas \ref{lem:h1} and \ref{lem:p}. So we can suppose that $f,k\geq 2$.

We recall that
\[
\ho{2}{\F{q}}{\Z}=\Extpow{2}{\Z}{\F{q}}
\]
and 
\[
\ho{3}{\F{q}}{\Z}\cong\left(\Extpow{3}{\Z}{\F{q}}\right)\oplus\left(\F{q}\otimes_\Z\F{q}\right)^\sigma
\]
where the second term denotes the subgroup fixed by the twist operator $\sigma$ (sending $x\otimes y$ to $y\otimes x$).

By Corollary \ref{cor:sus}, $\Inv{\Extpow{n}{\Z}{\F{q}}}{(\F{q}^\times)^2}=0$ unless $p=2$ and $f\leq n$ or $p>2$ and 
$(p-1)f\leq 2n$.  
Of course, $\Extpow{n}{\Z}{\F{q}}=0$ if $n>f$. Thus when $n=2$ we need only consider the cases $f=2$ and $p=2$ or $3$; i.e. 
$q=4$ or $9$.

We observe also $\Extpow{f}{\Z}{\F{q}}\cong \F{p}$ and that $x\in \F{q}^\times$ acts on this module as 
multiplication by $N_{\F{q}/\F{p}}(x)$.
  
Thus $\Inv{\Extpow{2}{\Z}{\F{4}}}{\F{4}^\times}=\Z/2$ since the norm map $\F{4}^\times\to\F{2}^\times$ is necessarily trivial.

Similarly $\Inv{\Extpow{2}{\Z}{\F{9}}}{(\F{9}^\times)^2}=\Z/3$ since $(\F{9}^\times)^2$ is the kernel of the norm map to $\F{3}^\times$.

This dispenses with the case $k=2$.

Now suppose $k=3$. By the remarks above, $\Inv{\Extpow{3}{\Z}{\F{q}}}{(\F{q}^\times)^2}=0$ unless $q=8$ or $27$. In both of these 
cases we obtain  $\Inv{\Extpow{3}{\Z}{\F{q}}}{(\F{q}^\times)^2}=\Z/p$ as in the case $n=2$. 

Again, by Corollary \ref{cor:sus}, $\Inv{(\F{q}\otimes\F{q})^\sigma}{(\F{q}^\times)^2}=0$ unless $q=4$ or $9$. We consider these 
cases individually:

$q=4$: Let $\F{4}=\F{2}(a)$ where $a^2=1+a$. 
Then $\left(\tens{2}{\F{2}}{F}\right)^\sigma$ is a $3$-dimensional $\F{2}$-space with 
basis $e_1:=1\otimes 1$, $e_2:=a\otimes a$ and $e_3:=1\otimes a+a\otimes 1$. If $\phi$ is the map induced by multiplication by 
$a$, then it has a $1$-dimensional $1$-eigenspace with basis $e_3$. Thus
\[
\psyl{\ho{0}{F^\times}{\ho{3}{\spl{2}{F}}{\Z}}}{p}\cong \Z/2.
\]

$q=9$: We let $\F{9}=\F{3}(i)$ where $i^2=-1$ and $\lambda:=1-i$ is an element of order $8$. 
 Then $(\tens{2}{\F{3}}{\F{9}})^\sigma$ is  the $3$-dimensional subspace of 
 $\tens{2}{\F{3}}{\F{9}}$ with basis $e_1=1\otimes 1$,  $e_2= i\otimes i$, and 
 $e_3=1\otimes i + i \otimes 1$. Then conjugation by $D(\lambda)$ induces 
 $\phi:=\mu_{\lambda^2}\otimes\mu_{\lambda^2}= \mu_i\otimes\mu_i$ on this space. 
 Since $\phi(e_1)=e_2$, $\phi(e_2)=e_1$ and $\phi(e_3)=2e_3$, it easily follows that 
  $\Inv{(\F{9}\otimes\F{9})^\sigma}{(\F{9}^\times)^2}$ is the $1$-dimensional subspace with basis 
  $e:=e_1+e_2$.  
 
This complete the proof of the lemma.
\end{proof}

\section{The third homology of $\spl{2}{F}$}\label{sec:main}

\subsection{Statement of the main theorem}
We recall two standard subgroups of $\spl{2}{F}$:
\[
T:=\left\{D(a)=\matr{a}{0}{0}{a^{-1}}\ |\ a\in F^\times\right\}\qquad 
B:=\left\{\matr{a}{b}{0}{a^{-1}}\ |\ a\in F^\times, b\in F\right\}
\]

\begin{lem}
\label{lem:hob}

If $F$ is an infinite field, then the inclusion $T\to B$ induces homology isomorphisms
\[
\ho{k}{T}{\Z}\cong\ho{k}{B}{\Z}
\]
for all $k\geq 0$.
\end{lem}
\begin{proof}
This is, for example, a special case of Lemma 9 in \cite{hutchinson:mat}.
\end{proof}

\begin{rem} 
For finite fields, the calculations of section \ref{sec:finite} show that this result fails in general. Thus, when 
$F$ is finite of characteristic $p$, we have 
\[
\ho{k}{B}{\Z}=\ho{k}{T}{\Z}\oplus\psyl{\ho{k}{B}{\Z}}{p}=\ho{k}{T}{\Z}\oplus\Inv{\ho{k}{F}{\Z}}{(F^\times)^2}
\]
and we tabulate the finite number of fields such that $\psyl{\ho{k}{B}{\Z}}{p}\not= 0$ for $k\leq 3$. 
\end{rem}

The rest of this section will to devoted to the proof of 

\begin{thm}\label{thm:main} Let $F$ be a field with at least $4$ elements.
\begin{enumerate}
\item
If $F$ is infinite, there is a natural complex 
\[
0\to\zhalf{\Tor{\mu_F}{\mu_F}}\to\ho{3}{\spl{2}{F}}{\zhalf{\Z}}\to\zhalf{\rbl{F}}\to 0.
\]
which is exact everywhere except possibly at the middle term. The middle homology is annihilated by $4$.

\item 
If $F$ is finite of odd characteristic, there is a complex 
\[
0\to\ho{3}{B}{\Z}\to\ho{3}{\spl{2}{F}}{\zhalf{\Z}}\to\zhalf{\rbl{F}}\to 0
\]
which is exact except possible at the middle term, where the homology has order at most $2$ .

\item 
If $F$ is finite of characteristic $2$, there is an exact sequence 
\[
0\to\ho{3}{B}{\Z}\to\ho{3}{\spl{2}{F}}{\zhalf{\Z}}\to\zhalf{\rbl{F}}\to 0.
\]
\end{enumerate}
\end{thm}

\subsection{Preliminaries}

Let $G$ be a group and let $P$ be a (left) $G$-set. We can use the action of $G$ on $P$ to study the homology of $G$ in the following way.
Let $X_n$ be the set consisting of ordered $n$-tuples of distinct points of $P$. We let $G$ act on $X_n$ via a diagonal action. Let
\[
C_n=C_n(P)=
\left\{
\begin{array}{ll}
\Z[X_n], &n\geq 1\\
\Z,&n=0
\end{array}
\right.   
\]

Let $d=d_n:C_n\to C_{n-1}$ be the $\Z[G]$-module homomorphism determined by
\[
d_n(x_1,\ldots,x_n)=\sum_{i=1}^n(-1)^{i+1}(x_1,\ldots,\hat{x_i},\ldots,x_n)
\]
for $n\geq 2$ and $d_1(x)=1$ for all $x\in P=X_1$. Then $C_\bullet=(C_n,d_n)$ is a complex of $\Z[G]$-modules. It is almost 
acyclic:
\begin{lem}\label{lem:acyclic}
Suppose that the set $P$ has cardinality $c$.  Then $H_n(C)=0$ if $n\not=c$.

In particular, $C_\bullet$ is acyclic if $P$ is infinite.
\end{lem} 
\begin{proof}
If $S\subset P$, we will let $D_n(S)$ denote the subgroup of $C_n$ generated by those $n$-tuples in which 
all elements of $S$ occur. Thus $D_n(S)=0$ if $S$ has more than $n$ elements, and $D_n(S_1\cup S_2)=D_n(S_1)\cap D_n(S_2)$.

For $x\in P$ we define $\Z[G]$-homomorphisms $S_x:C_n\to C_{n+1}$ by 
\[
S_x(x_1,\ldots,x_n)=
\left\{
\begin{array}{ll}
(x,x_1,\ldots,x_n),&\mbox{ if } x\not\in\{ x_1,\ldots,x_n\}\\
0,&\mbox{ otherwise }
\end{array}
\right.
\]

Thus, if $(x_1,\ldots,x_n)\in X_n$ and $x\not\in\{ x_1,\ldots,x_n\}$ then 
\begin{eqnarray*}
dS_x(x_1,\ldots,x_n)=d(x,x_1,\ldots,x_n)=(x_1,\ldots,x_n)-S_xd(x_1,\ldots, x_n)
\end{eqnarray*}

On the other hand, if $x=x_j$ for some $j$, then 
\[
S_x(d(x_1,\ldots,x_n))=(-1)^{j+1}(x_j,x_1,\ldots,\hat{x_j},\ldots,x_n)
\]
and thus
\begin{eqnarray*}
0=d(S_x(x_1,\ldots,x_n)=(x_1,\ldots,x_n)-S_x(d(x_1,\ldots,x_n))
-\left\{ (x_1,\ldots,x_n)+(-1)^{j}(x_j,x_1,\ldots,x_n)\right\}
\end{eqnarray*}

Either way, whether $x$ belongs to $\{x_1,\ldots,x_n\}$ or not, we have
\[
dS_x(x_1,\ldots,x_n)=(x_1,\ldots,x_n)-S_x(d(x_1,\ldots,x_n))+w
\]
where $w\in D_n(\{ x\})$. Furthermore, if $(x_1,\ldots,x_n)\in D_n(S)$, then 
$w\in D_n(S\cup \{ x\})$.

Now suppose that $x_1,\ldots,x_{n+1}$ are $n+1$ distinct elements of $P$. Let $z\in C_n$ be a cycle. 
Then
\[
(dS_{x_1}-\id{})z=S_{x_1}(dz)+z_1=z_1
\] 
where $z_1$ is a cycle and $z_1\in D_n(\{ x_1\})$. 

Thus $(dS_{x_2}-\id{})(z_1)=z_2$ where $z_2$ is a cycle in $D_n(\{ x_1,x_2\})$.  Repeating the process, we get 
\[
(dS_{x_{n+1}}-\id{})(dS_{x_n}-\id{})\cdots(dS_{x_1}-\id{})(z)\in D_n(\{ x_1,\ldots,x_{n+1}\})=0.
\]

This has the form $dy+(-1)^{n+1}(z)=0$ and thus $z=d((-1)^ny)$ is a boundary, as required.
\end{proof}
\begin{rem}
If $P$ is finite of size $c\geq 2$, then it is easy to see that $H_c(C)\not=0$; in fact, a straightforward Euler characteristic 
calculation shows that it is a free abelian group of 
rank
\[
c!\left(\frac{1}{2!}-\frac{1}{3!}+\cdots +(-1)^{c}\frac{1}{c!}\right).
\]
\end{rem}

Let now $L_\bullet=L_\bullet(P)$ be the complex defined by 
\[
L_n:=
\left\{
\begin{array}{ll}
C_{n+1}, & n\geq 0\\
0, & n<0
\end{array}
\right.
\]
If $P$ is infinite, Lemma \ref{lem:acyclic} shows that $L_\bullet$ is weakly equivalent to the $\Z$ 
(considered as a complex concentrated in dimension $0$) and more 
generally, if $P$ has cardinality $c$ then $H_n(L)=0$ for $n\not=0,c-1$ and $H_0(L)\cong \Z$.  

\begin{lem}\label{lem:spectral}
Let $L_\bullet$ be a complex of $\Z[G]$-modules and suppose that $H_n(L)=0$ for $1\leq n\leq k$. Then 
\[
\ho{n}{G}{L_\bullet}=\ho{n}{G}{\hoz{0}{L}}\mbox{ for } 0\leq n\leq k.
\] 
\end{lem}
\begin{proof}
Recall (see, for example, \cite{brown:coh}, VII.5) that $\ho{n}{G}{L_\bullet}$ is by definition the homology of the 
total complex  
\[
B_\bullet\otimes_{\Z[G]}L_\bullet.
\]
where $B_\bullet$ is a (right) projective resolution of $\Z$ over $\Z[G]$.

This is the total complex of a bounded double complex and there are thus two filtrations and two associated spectral sequences 
converging to $\ho{n}{G}{L_\bullet}$. The first takes the form 
\[
E^2_{p,q}=\ho{p}{G}{\hoz{q}{L}}\Longrightarrow \ho{p+q}{G}{L_\bullet},
\] 
with differentials $d^r:E^r_{p,q}\to E^r_{p-r,q+r-1}$. 

By our assumptions, $E^2_{p,q}=0$ for $1\leq q\leq k$. In particular, all higher differentials leaving $E^r_{p,0}, 0\leq p\leq k$ are 
$0$, so that $E^\infty_{p,0}=E^2_{p,0}= \ho{p}{G}{\hoz{0}{L}}$ for $p\leq k$, and there are no other nonzero terms in (total) 
 dimension at most $k$.
\end{proof}

In particular, if $P$ has cardinality $c$, then $\ho{n}{G}{L_\bullet(P)}=\ho{n}{G}{\Z}$ for $n\leq c-2$.  

The second spectral sequence for $\ho{n}{G}{L_\bullet}$ has the form 
\[
E^1_{p,q}=\ho{p}{G}{L_q}\Longrightarrow \ho{p+q}{G}{L_\bullet},\quad d^r:E^r_{p,q}\to E^r_{p+r-1,q-r}.
\] 
The map $d^1:E^1_{p,q}=\ho{p}{G}{L_q}\to \ho{p}{G}{L_{q-1}}=E^1_{p,q-1}$ is just the map induced by $d_q:L_q\to L_{q-1}$. 

Thus, by Lemma \ref{lem:spectral}, when $L_\bullet=L_\bullet(P)$ for a $G$-set $P$ of cardinality $c$,  we have a spectral sequence 
with $E^1_{p,q}=\ho{p}{G}{L_q}$ whose abutment in dimensions less than $c-2$ is $\ho{n}{G}{\Z}$.


We now apply 
this set-up to the particular case $G=\spl{2}{F}$ and $P=\projl{F}$ (the resulting spectral sequence has been studied elsewhere; 
for example in \cite{mazz:sus}). If $F$ has $q$ elements, then $\projl{F}$ has $q+1$ elements 
and thus we have a spectral sequence which abuts to $\ho{k}{\spl{2}{F}}{\Z}$ for $k\leq q-1$. Since we wish to use the spectral 
sequence to calculate $\ho{3}{\spl{2}{F}}{\Z}$, we will require that $q\geq 4$; i.e. 
in the spectral sequence arguments  below $F$ is a field with at least $4$ elements. 
 

\subsection{Module structure on the spectral sequence} 
Our spectral sequence has a natural graded module structure (in a sense to be detailed) 
over the Pontryagin ring $\ho{\bullet}{\mu_2}{\Z}$ which facilitates the 
calculation of some higher differentials. More generally, we have the following  situation:

Let $G$ be a group and let $H$ be a subgroup of the centre, $Z(G)$, of $G$. The integral homology of $G$ is a graded module for the 
Pontryagin ring $\ho{\bullet}{H}{\Z}$ of the abelian group $H$ (see, for example, Brown \cite{brown:coh}, Chapter V): 
Let $B_\bullet$ (respectively $B'_\bullet$) be a right projective 
resolution of $\Z$ over $\igr{G}$ (respectively $\igr{H}$). Then $B'\otimes B$ is a projective resolution of $\Z$ over $\igr{H\times G}$.
Let
\[
\tau:B'\otimes_\Z B\to B
\]
be a map an augmentation-preserving chain map compatible with the group homomorphism $H\times G\to G, (h,g)\mapsto h\cdot g$. Then 
the induced composite map 
\[
(B'\otimes_{\igr{H}}\Z)\otimes(B\otimes_{\igr{G}}\Z)\to (B'\otimes B)\otimes_{\igr{H\times G}}\Z \to B\otimes_{\igr{G}}\Z
\] 
induces the required homomorphisms
\[
\ho{k}{H}{\Z}\otimes\ho{p}{G}{\Z}\to \ho{k+p}{G}{\Z}
\]
which define the module structure.

Now suppose that  $C_\bullet$ is a complex of $\igr{G}$-modules which is weakly equivalent to $\Z$, considered as complex concentrated 
in dimension $0$.  Then, as noted, we have a spectral sequence abutting to $\ho{\bullet}{G}{\Z}$ associated to the double complex
\[
D_{p,q}=B_\bullet\otimes_{\igr{G}}C_\bullet.
\]
If we further assume that $H$ acts trivially on the complex $C_\bullet$, then, using $\tau$, we obtain (replacing $\Z$ by 
$C_\bullet$ above) a map of double complexes
\[
(B'_\bullet\otimes_{\igr{H}}\Z)\otimes(B_\bullet\otimes_{\igr{G}}C_\bullet)\to 
(B'_\bullet\otimes B_\bullet)\otimes_{\igr{H\times G}}C_\bullet \to B_\bullet\otimes_{\igr{G}}C_\bullet
\]
which induces, for all $r\geq 1$ 
maps
\[
\ho{k}{H}{\Z}\otimes E^r_{p,q}\to E^r_{k+p,q}
\]
such that the diagrams 
\[
\xymatrix{
\ho{k}{H}{\Z}\otimes E^r_{p,q}\ar[r]\ar[d]^-{(-1)^k\otimes d^r}& E^r_{k+p,q}\ar[d]^-{d^r}\\
\ho{k}{H}{\Z}\otimes E^r_{p+r-1,q-r}\ar[r]&E^r_{k+p+r-1,q-r}
}
\]
commute; i.e. we have $d^r(\alpha\cdot z)=(-1)^k\alpha\cdot d^r(z)$ for $\alpha\in\ho{k}{H}{\Z}$, $z\in E^r_{p,q}$. 

\subsection{The $E^1$-page of the spectral sequence}

Let $X_n$ denote the set of ordered $n$-tuples of distinct points of $\projl{F}$ and $L_n=\zhalf{\Z}X_{n+1}$.  Thus there is a spectral 
sequence of the form 
\[
E^1_{p,q}=\ho{p}{\spl{2}{F}}{L_q}\Longrightarrow \ho{p+q}{\spl{2}{F}}{\zhalf{\Z}}
\]
derived from the double complex 
\[
E^0_{\bullet,\bullet}=B_\bullet\otimes_{\zhalf{\Z}[\spl{2}{F}]} L_\bullet
\]
where $B_\bullet$ is the standard (right) bar resolution of $\Z$ over $\spl{2}{F}$, tensored with $\zhalf{\Z}$.

Let $\delta:F^\times \to \gl{2}{F}$ be the map $a\mapsto \mathrm{diag}(a,1)$ (a splitting of the determinant map). 

Let $F^\times$ act on $E^0_{p,q}$ by
\[
a\cdot\left([g_1|\cdots|g_p]\otimes (x_0,\ldots,x_q)\right)=[\delta(a)g_1\delta(a)^{-1}|\cdots |\delta(a)g_p\delta(a)^{-1}]
\otimes \delta(a)\cdot (x_0,\ldots,x_q).
\]

This action makes $E^0_{\bullet,\bullet}$ into a double complex of $\zhalf{\sgr{F}}$-modules and the induced actions on 
$E^1_{p,q}=\ho{p}{\spl{2}{F}}{L_q}$ are the natural actions derived from the $\gl{2}{F}$-action on $L_q$ and the short exact sequence
$1\to \spl{2}{F}\to\gl{2}{F}\to F^\times\to 1$.

The $E^1$-page of the spectral sequence is easily calculated using Shapiro's Lemma since the $\spl{2}{F}$-modules $L_n$ are permutation 
modules, and hence induced modules.

Thus, $\spl{2}{F}$ acts transitively on $X_1=\projl{F}$ and the stabilizer of $(\infty)$ is $B$. Thus 
\[
L_0=\zhalf{\Z}[X_1]\cong\zhalf{\Z}[B\backslash\spl{2}{F}]\cong\Ind{\zhalf{\Z}[B]}{\zhalf{\Z}[\spl{2}{F}]}{\zhalf{\Z}}
\]
so that 
\[
E^1_{p,0}=\ho{p}{\spl{2}{F}}{L_0}\cong\ho{p}{B}{\zhalf{\Z}}
\] 
by Shapiro's Lemma.

Similarly, $\spl{2}{F}$ acts transitively on $X_2$ and the stabilizer of $(0,\infty)$ is $T$. So 
\[
E^1_{p,1}=\ho{p}{\spl{2}{F}}{L_1}\cong\ho{p}{T}{\zhalf{\Z}}.
\]

For $n\geq 3$ the stabilizer of an element $(x_1,\ldots,x_n)$ in $\spl{2}{F}$ is $\mu_2(F)=Z(\spl{2}{F})$. Using 
Corollary \ref{cor:xnsl} above, it follows that for $q\geq 2$, and when the characteristic of $F$ is not $2$, we have 
\[
E^1_{p,q}=\sgr{F}[Z_{q-2}]\otimes\ho{p}{\mu_2}{\Z}\cong\left\{
\begin{array}{ll}
\sgr{F}[Z_{q-2}],&p=0\\
0,& p>0 \mbox{ even}\\
\sgr{F}[Z_{q-2}]\otimes\Z/2,& p>0 \mbox{ odd}\\
\end{array}
\right.
\] 
where $Z_n$ is the set of ordered $n$-tuples $[z_1,\ldots,z_n]$ of distinct points of $\projl{F}\setminus\{\infty,0,1\}$. When 
the characteristic is $2$, of course, $\mu_2(F)=\{ 1\}$ and $E^1_{p,q}=0$ whenever $p\geq 1$ and $q\geq 2$. 

Note also 
that for $q\geq 2$, the module structure mentioned above is reflected in the tensor product decomposition of the terms; if 
$\alpha\in\ho{p}{\mu_2}{\Z}$ and $z\in E^1_{p,q}=\sgr{F}[Z_{q-2}]$, then 
\[
\alpha\cdot z= z\otimes\alpha \in \sgr{F}[Z_{q-2}]\otimes\ho{k}{\mu_2}{\Z}=E^1_{p,q}.
\] 

Thus our $E^1$-page has the form
\begin{eqnarray*}
\xymatrix{
\vdots&\vdots&\vdots&\vdots&\\
\zhalf{\sgr{F}}[Z_2]\ar[d]^-{d^1}&\sgr{F}[Z_2]\otimes \mu_2\ar[d]^-{d^1}
&\vdots&\vdots&\hdots\\
\zhalf{\sgr{F}}[Z_1]\ar[d]^-{d^1}&\sgr{F}[Z_1]\otimes\mu_2\ar[d]^-{d^1}&0&\sgr{F}[Z_1]\otimes\Z/2\ar[d]^-{d^1}&\hdots\\
\zhalf{\sgr{F}}\ar[d]^-{d^1}&\sgr{F}\otimes\mu_2\ar[d]^-{d^1}&0&\sgr{F}\otimes\Z/2\ar[d]^-{d^1}&\hdots\\
\zhalf{\Z}\ar[d]^-{d^1}&\ho{1}{T}{\zhalf{\Z}}\ar[d]^-{d^1}&\ho{2}{T}{\zhalf{\Z}}\ar[d]^-{d^1}&\ho{3}{T}{\zhalf{\Z}}\ar[d]^-{d^1}&\hdots\\
\zhalf{\Z}&\ho{1}{B}{\zhalf{\Z}}&\ho{2}{B}{\zhalf{\Z}}&\ho{3}{B}{\zhalf{\Z}}&\hdots
}
\end{eqnarray*}
when the characteristic of $F$ is not $2$.

\subsection{ The  $E^2$-page}

By the calculations of section \ref{sec:bloch} above the differential
\[
d^1:E^1_{0,4}=\zhalf{\sgr{F}}[Z_2]\to\zhalf{\sgr{F}}[Z_1]=E^1_{0,3}
\]
is given by
\[
[x,y]\mapsto S_{x,y}=\gpb{x}-\gpb{y}+\an{x}\gpb{\frac{y}{x}}-\an{x^{-1}-1}
\gpb{\frac{1-x^{-1}}{1-y^{-1}}}+\an{1-x}\gpb{\frac{1-x}{1-y}}.
\]
Thus $E^1_{0,3}/\image{d^1}:=\zhalf{\rpb{F}}$.

On the other hand,  for $x\in Z_1$ we have 
\begin{eqnarray*}
d^1(\gpb{x})=d((0,\infty,1,x))=(\infty,1,x)-(0,1,x)+(0,\infty,x)-(0,\infty,1)
\end{eqnarray*}
which corresponds to the element
\begin{eqnarray*}
&\an{\phi(\infty,1,x)}-\an{\phi(0,1,x)}+\an{\phi(0,\infty,x)}-\an{\phi(0,\infty,1)}\\
=&\an{1-x}-\an{x(1-x)}+\an{x}-\an{1}=-\pf{x}\pf{x(x-1)}\\
=&-\pf{1-x}\pf{x}=-\lambda_1(\gpb{x})\in \zhalf{\sgr{F}}=E^1_{0,2}.
\end{eqnarray*}

Thus $E^2_{0,3}=\rpbker{F}:=\ker{\lambda_1:\zhalf{\rpb{F}}\to\aug{F}^2}$. 


Using the module structure, the map 
\[
d^1:E^1_{1,3}=\sgr{F}[Z_1]\otimes\mu_2 \to \sgr{F}\otimes\mu_2=E^1_{1,2}
\]
is the map $[x]\otimes z\mapsto -\lambda_1([x])\otimes z$. Thus 
\[
\frac{E^1_{1,2}}{\image{d^1}}=\frac{\sgr{F}}{\gwrel{F}}
\otimes\mu_2=\gw{F}\otimes\mu_2
\]

Similarly, the map $d^1:E^1_{0,2}=\sgr{F}\to E^1_{0,1}=\Z$ is easily seen to be the 
natural augmentation homomorphism sending $\an{x}\in\sq{F}$ to $1$, and hence the differential
\[
d^1:E^1_{1,2}=\sgr{F}\otimes\mu_2\to F^\times\cong\ho{1}{T}{\Z}=E^1_{1,1}
\]
sends $\an{x}\otimes z$ to $z$ (for all $x\in F^\times$, $z\in\mu_2\subset F^\times$). 
It follows that $E^2_{1,2}= \gwaug{F}\otimes \mu_2$.

Similarly, we obtain that $E^2_{1,3}=\rpbker{F}\otimes\mu_2$ (keeping in mind that all the groups $\sgr{F}[Z_i]$ 
are $\Z$-free). 

Now let 
\[
w:=\matr{0}{-1}{1}{0}\in\spl{2}{F}.
\]
Then $w(\infty)=0$ and $w(0)=\infty$. It follows easily that the differential
\[
d^1:E^1_{p,1}=\ho{p}{T}{\zhalf{\Z}}\to \ho{p}{B}{\zhalf{\Z}}= E^1_{p,0}
\]
is the composite
\begin{eqnarray*}
\xymatrix{
\ho{p}{T}{\zhalf{\Z}}\ar[r]^-{w_p-1}
&
\ho{p}{T}{\zhalf{\Z}}\ar[r]
&\ho{p}{B}{\zhalf{\Z}}
}
\end{eqnarray*}
where $w_p:\ho{p}{T}{\zhalf{\Z}}\to\ho{p}{T}{\zhalf{\Z}}$ is the map induced by conjugation by $w$. 
However, conjugating by $w$ 
is just the inversion map on $T\cong F^\times$. For future 
convenience, we will define
\[
A_i(F)=\left\{
\begin{array}{ll}
0,& F\mbox{ is infinite}\\
\psyl{\ho{i}{B}{\Z}}{p}=\psyl{\ho{i}{\spl{2}{F}}{\Z}}{p},& F\mbox{ is finite of characteristic $p$}. 
\end{array}
\right.
\] 

Thus $d^1=w_1-1:E^1_{1,1}=F^\times\to E^1_{1,0}=\ho{1}{B}{\Z}=F^\times\oplus A_1(F)$ is the map $x\mapsto x^{-2}$. It 
follows that $E^2_{1,0}=\sq{F}\oplus A_1(F)$. 

Furthermore, $w_2$ is the identity map on $\ho{2}{T}{\zhalf{\Z}}=F^\times\wedge F^\times$ 
and hence $d^1:E^1_{2,1}\to E^1_{2,0}$ 
is the zero map. So $E^2_{2,0}=E^3_{2,0}=\ho{2}{B}{\zhalf{\Z}}= \Extpow{2}{}{F^\times}\oplus A_1(F)$.

Recall that 
\[
E^1_{3,1}=\ho{3}{T}{\Z}\cong\ho{3}{F^\times}{\Z}\cong\Extpow{3}{\Z}{F^\times}\oplus\Tor{\mu_F}{\mu_F}
\]
and 
\[
E^1_{3,0}=\ho{3}{B}{\Z}\cong\ho{3}{T}{\Z}\oplus A_3(F).
\]
The the map $d^1:E^1_{3,1}\to E^1_{3,0}$ restricts to the identity on the factors $\Extpow{2}{}{F^\times}$ and to the zero 
map on $\Tor{\mu_F}{\mu_F}$. It follows that $E^2_{3,0}=\Tor{\mu_F}{\mu_F}\oplus A_3(F)$.

Thus the relevant part of the $E^2$-page has the form
\begin{eqnarray*}
\xymatrix{
\rpbker{F}\ar[ddr]^-{d^2}&\rpbker{F}\otimes\mu_2\ar[ddr]^-{d^2}&0&\vdots\\
\gwaug{F}\ar[ddr]^-{d^2}&\gwaug{F}\otimes\mu_2\ar[ddr]^-{d^2}&0&\vdots\\
0&0&\Extpow{2}{\Z}{F^\times}&\vdots\\
\Z&\sq{F}\oplus A_1(F)&\Extpow{2}{\Z}{F}\oplus A_2(F)&\Tor{\mu_F}{\mu_F}\oplus A_3(F)
}
\end{eqnarray*}

\subsection{The $E^3$-page}
We begin by observing that since the edge homomorphisms 
$A_i(F)\to \ho{i}{\spl{2}{F}}{\Z}$ are necessarily injective, it follows that the base terms $E^r_{p,0}$ always factor 
in the form $G^r_{p,0}\oplus A_p(F)$ and that any differential $d^r$ with target $E^r_{p,0}$ has image contained in 
$G^r_{p,0}$.  

Now the differential $d^2:E^2_{0,2}=\gwaug{F}\to \sq{F}\subset E^2_{2,0}$ has been calculated by Mazzoleni 
(\cite{mazz:sus}, Lemma 5): it sends $\pf{x}$ to $\an{x}$ (for $x\not= 1$). 

If $\an{x}\otimes -1\in \gwaug{F}\otimes\mu_2=E^2_{1,2}$, it follows - using the module structure on the spectral 
sequence - that
\[
d^2(\an{x}\otimes -1)=-d^2(\an{x})\cdot -1=x\wedge -1 \in F^\times\wedge F^\times \subset E^2_{2,0}.
\] 
(Here we use the fact that under the identification $F^\times\wedge F^\times=\ho{2}{F^\times}{\Z}$, the wedge 
product corresponds to the Pontryagin product on homology).

Thus $E^3_{2,0}=\sextpow{2}{\Z}{F^\times}\oplus A_2(F)$ where we set
\[
\sextpow{2}{\Z}{F^\times}:=\frac{F^\times\wedge F^\times}{F^\times\wedge \mu_2}.
\]

If $F$ is a finite field $F^\times\wedge F^\times=0$ and $E^3_{1,2}=E^\infty_{1,2}$ is a quotient of 
$\gwaug{F}\otimes\mu_2\cong\mu_2(F)$. Thus $E^\infty_{1,2}$ has order at most $2$ if $F$ is finite of odd characteristic, and 
is $0$ if $F$ is finite of characteristic $2$. 

In any case, for any field $F$, the term $E^3_{1,2}=E^\infty_{1,2}$ is annihilated by $2$.  

Of course, the differential $d^2:E^2_{3,0}=\rpbker{F}\to E^2_{1,1}=0$ is necessarily the zero map, and it follows, using 
the module structure again, that the differential $d^2:E^2_{3,1}=\rpbker{F}\otimes\mu_2\to \Extpow{2}{}{F^\times}=E^2_{2,2}$ 
is also the zero map.   

Thus the relevant part of the $E^3$-page has the form
\begin{eqnarray*}
\xymatrix{
E^3_{0,4}\ar[dddrr]^-{d^3}&\vdots&\vdots&\vdots\\
\rpbker{F}\ar[dddrr]^-{d^3}&\vdots&\vdots&\vdots\\
\gwaug{F}^2&E^3_{1,2}&\vdots&\vdots\\
0&0&\Extpow{2}{\Z}{F^\times}&\vdots\\
\Z&A_1(F)&\sextpow{2}{\Z}{F}\oplus A_2(F)&\Tor{\mu_F}{\mu_F}\oplus A_3(F)
}
\end{eqnarray*}

\subsection{The $E^4$-page}
We begin by calculating the differential
 \[
d^3:E^3_{0,3}=\zhalf{\rpbker{F}}\to \sextpow{2}{\Z}{F^\times}
\subset E^3_{2,0}.
\]

Let 
\[
E(GL)^1_{p,q}=\ho{p}{\gl{2}{F}}{L_q}\Longrightarrow \ho{p+q}{\gl{2}{F}}{\zhalf{\Z}}
\]
be the spectral sequence derived form the the action of $\gl{2}{F}$ on the complex $L_\bullet$ and converging to the integral homology of 
$\gl{2}{F}$. (This spectral has been studied in \cite{hutchinson:mat}.) 
Then the inclusion $\spl{2}{F}\to\gl{2}{F}$ induces a map of spectral sequences $E^r_{p,q}\to E(GL)^r_{p,q}$.  
Now $E(GL)^3_{0,3}=\zhalf{\pb{F}}$ and $E(GL)^3_{2,0}=\ho{2}{\tilde{T}}{\Z}/(w_2-1)$, where $\tilde{T}\subset\gl{2}{F}$ is the 
subgroup consisting of all diagonal matrices. There is a split-exact sequence (\cite{hutchinson:mat}, Lemma 4)  
\begin{eqnarray*}
\xymatrix{
0\ar[r]
&\asym{2}{\Z}{F^\times}\ar[r]
& \ho{2}{\tilde{T}}{\zhalf{\Z}}/(w_2-1)\ar[r]^-{\mathrm{det}}
&\ho{2}{F^\times}{\zhalf{\Z}}\ar[r]
&0  
}
\end{eqnarray*}

Now the image of $d^3:E(GL)^3_{0,3}=\zhalf{\pb{F}}\to E(GL)^3_{2,0}$ factors through the term $\asym{2}{\Z}{F^\times}$ and is given by the formula
\[
\zhalf{\pb{F}}\to\asym{2}{\Z}{F^\times},\qquad \gpb{x}\mapsto 
(1-x)\otimes x.
\]
(See \cite{hutchinson:mat}, p190, and allow for the fact that the term 
$[x]\in\pb{F}$ in this paper corresponds to $\gpb{1/x}$ there.) 

We observe that, for any field $F$ there is a natural injective homomorphism
\[
\frac{F^\times\wedge F^\times}{F^\times\wedge \mu_2} =\sextpow{2}{\Z}{F^\times}
\to \asym{2}{\Z}{F^\times}, a\wedge b\mapsto 
2(a\asymm b).
\]

It is easily seen that the inclusion $T\to \tilde{T}$ induces the map 
\begin{eqnarray*}
&\extpow{2}{F^\times}=\ho{2}{F^\times}{\zhalf{\Z}}\cong\ho{2}{T}{\zhalf{\Z}}
\to \asym{2}{\Z}{F^\times}
\subset   \ho{2}{\tilde{T}}{\Z}/(w_2-1)\\
&a\wedge b\mapsto 2(a\otimes b).
\end{eqnarray*}
and thus induces the injection
 $\sextpow{2}{\Z}{F^\times}\to \asym{2}{\Z}{F^\times}$.

Putting all of this together we get the commutative diagram
\begin{eqnarray*}
\xymatrix{
\rpbker{F}\ar[r]^-{d^3}\ar[d]\ar[dr]^-{\lambda_2}
&
\sextpow{2}{\Z}{F^\times}\ar@{^{(}->}[d]\\
\pb{F}\ar[r]^-{\lambda}
&
\asym{2}{\Z}{F^\times}
}
\end{eqnarray*}
from which it follows that 
\[
E^\infty_{0,3}=E^4_{0,3}=\ker{d^3}=
\ker{\lambda_2:\rpbker{F}\to \asym{2}{\Z}{F^\times}}=\rbl{F}. 
\]


Finally, we will show that 
$2\cdot E^\infty_{2,1}=2\cdot E^4_{2,1}=0$.  
In order to this we will need a \emph{technical lemma}:

Let $G$ be a group and $L_\bullet$ a complex of $\Z[G]$-modules concentrated in non-negative dimensions. Let $L_\bullet(m)$ denote the 
truncated complex
\[
L_k(m):=\left\{
\begin{array}{ll}
L_k,&k\geq m\\
0,&k<m\\
\end{array}
\right.
\] 
(So $L_\bullet=L_\bullet(0)$.) 

Consider the spectral sequences 
\[
E^1(m)_{p,q}=\ho{p}{G}{L_q(m)}\Longrightarrow\ho{p+q}{G}{L_\bullet(m)}.
\]
If $m'\geq m$, the natural map of complexes $L(m')\to L(m)$ induces a map of spectral sequences 
\[
E^r(m')_{p,q}\to E^r(m)_{p,q}
\]
compatible with the map on abutments $\ho{p+q}{G}{L_\bullet(m')}\to\ho{p+q}{G}{L_\bullet(m)}$.

Note that $E^r(m)_{p,q}=0$ for $q<m$ and thus $E^r(m)_{0,q}=E^\infty(m)_{0,q}$ for $r>q-m$. Similarly, the if $m'\geq m$ 
then $E^r(m')_{p,q}=E^r(m)_{p,q}$ as long as $r\leq q-m'+1$.  

If $A$ is a $\Z[G]$-module, we let $A[m]$ denote the module $A$ considered as a complex concentrated in dimension $m$. Observe, in 
particular, that the differential $d:L_{m+1}\to L_m$ induces a map of complexes $\phi:L_\bullet(m+1)\to L_m[m+1]$. Our technical lemma then 
states:

\begin{lem}\label{lem:tech}
The following diagram commutes for any $r\geq 1$, $m\geq 0$
\begin{eqnarray*}
\xymatrix{
\ho{r+m}{G}{L_\bullet(m+1)}\ar[r]^-{\phi}\ar@{>>}[d]
&\ho{r+m}{G}{L_m[m+1]}\ar[d]^-{=}\\
E^\infty(m+1)_{0,r+m}\ar[d]^-{=}
&\ho{r-1}{G}{L_m}\ar[d]^-{=}\\
E^r(m+1)_{0,r+m}\ar[d]^-{=}
&E^1(m)_{r-1,m}\ar@{>>}[d]\\
E^r(m)_{0,r+m}\ar[r]^{d^r}
&E^r(m)_{r-1,m}
}
\end{eqnarray*}
\end{lem}

\begin{proof} This is a tedious but straightforward verification from the definitions (it is clearly enough to consider
the case $m=0$).
\end{proof}

Applying this to the group $G=\spl{2}{F}$ and the 
complex $L_\bullet=L_\bullet(\projl{F})$ in the case $r=3$ and $m=1$ gives 
a commutative diagram
\begin{eqnarray*}
\xymatrix{
\ho{4}{\spl{2}{F}}{L_\bullet(2)}\ar[r]^-{\phi}\ar[d]^-{=}
&\ho{4}{\spl{2}{F}}{L_1[2]}\ar[d]^-{=}\\
E^3(1)_{0,4}\ar[d]^-{=}
&\ho{2}{\spl{2}{F}}{L_1}\ar[d]^-{=}\\
E^3_{0,4}\ar[r]^-{d^3}
&E^3_{2,1}
}
\end{eqnarray*}

Now let $W_k=\ker{L_k\to L_{k-1}}$, so that (in sufficiently low dimensions) we have short exact sequences 
\[
0\to W_k\to L_k\to W_{k-1}\to 0
\]
by Lemma \ref{lem:acyclic}.  The map $d: L_2 \to W_1$ induces a map of complexes $L_\bullet(2)\to W_1[2]$ which is a weak 
equivalence in low dimensions. In particular, it induces an isomorphism 
\[
\ho{4}{\spl{2}{F}}{L_\bullet(2)}\cong \ho{4}{\spl{2}{F}}{W_1[2]}=\ho{2}{\spl{2}{F}}{W_1}.
\]

Putting these facts together gives us:

\begin{cor}
There is a commutative diagram
\begin{eqnarray*}
\xymatrix{
\ho{2}{\spl{2}{F}}{W_1}
\ar[r]\ar[d]^-{\cong}
&\ho{2}{\spl{2}{F}}{L_1}
\ar[d]^-{\cong}\\
E^3_{0,4}\ar[r]^-{d^3}
& E^3_{2,1}
}
\end{eqnarray*}
where the top horizontal map is induced by the inclusion $W_1\to L_1$.
\end{cor}

Thus from this corollary and the long exact homology sequence of the exact sequence 
\[
0\to W_1\to L_1\to W_0\to 0
\]
it follows that the image of $d^3:E^3_{0,4}\to E^3_{2,1}=\ho{2}{T}{\zhalf{\Z}}$ is equal to the kernel of 
the map, $\mu$ say, 
\[
\extpow{2}{F^\times}\cong\ho{2}{T}{\zhalf{\Z}}=\ho{2}{\spl{2}{F}}{L_1}
\to\ho{2}{\spl{2}{F}}{W_0}.
\]

We will show that the map $2\cdot \mu$ is zero:

Let $B_\bullet \to \Z$ be a projective resolution of $\Z$ as a $\Z[\spl{2}{F}]$-module. So $\ho{\bullet}{T}{\zhalf{\Z}}$ 
is the homology of the complex $B_\bullet\otimes_{\Z[T]}\zhalf{\Z}$, and the composite
\[
\ho{\bullet}{T}{\zhalf{\Z}}\to\ho{\bullet}{\spl{2}{F}}{L_1}\to \ho{\bullet}{\spl{2}{F}}{W_0}
\]
is described on the level of chains 
\[
B_\bullet\otimes_{\Z[T]}\zhalf{\Z}\to B_\bullet\otimes_{\Z[\spl{2}{F}]}L_1\to B_\bullet\otimes_{\Z[\spl{2}{F}]}W_0
\]
by
\[
\gamma\otimes 1\mapsto \gamma\otimes (\infty,0)\mapsto \gamma\otimes((0)-(\infty)).
\]

Recall that 
\[
w=\matr{0}{-1}{1}{0}
\]
acts on $T$ by conjugation, and the action of $w$ on the homology of $T$ is described on the level of chains by
\[
\gamma\otimes 1\mapsto \gamma\cdot w^{-1}\otimes 1. 
\]

Thus, if $z\in\ho{2}{T}{\zhalf{\Z}}$ is represented by $\gamma\otimes 1$, then $w_2\cdot z$ is represented by 
$\gamma\cdot w^{-1}\otimes 1$. Thus
\begin{eqnarray*}
&\mu(w_2\cdot z)=\mu(\gamma\cdot w^{-1}\otimes 1)=\gamma\cdot w^{-1}\otimes ((0)-(\infty))\\
&=\gamma\otimes w^{-1}\cdot((0)-(\infty))=\gamma\otimes ((\infty)-(0))=-\mu(z).
\end{eqnarray*}

Since, as observed above, $w_2$ is the identity map, we have $2\mu(z)=0$ 
as required. It follows that $2\cdot E^\infty_{2,1}=0$.

\subsection{The calculation of $\ho{3}{\spl{2}{F}}{\Z}$}

Now the map 
$\Tor{\mu_F}{\mu_F}\to\ho{3}{\spl{2}{F}}{\Z}$ is injective, since, 
for example, the composite 
\[
\Tor{\mu_F}{\mu_F}\to\ho{3}{\spl{2}{F}}{\Z}\to\ho{3}{\gl{2}{F}}{\Z}
\]
is injective when $F$ is infinite by the results of Suslin 
(\cite{sus:bloch}), while for finite fields we have shown that the map 
$\Tor{\mu_F}{\mu_F}=\ho{3}{T}{\Z}\to\ho{3}{\spl{2}{F}}{\Z}$ is injective in 
section \ref{sec:finite} above. 
It thus follows
that 
\[
E^\infty_{3,0}=\Tor{\mu_F}{\mu_F}\oplus A_3(F)=
\left\{
\begin{array}{ll}
\Tor{\mu_F}{\mu_F},& F\mbox{ infinite}\\
\ho{3}{B}{\Z},& F\mbox{ finite}
\end{array}
\right.
\]

Now, by the computations above, $2\cdot E^\infty_{1,2}=2\cdot E^\infty_{2,1}$ for 
any field $F$. Furthermore, clearly $E^\infty_{2,1}=0$ for any finite field $F$ 
since $E^1_{1,2}\cong F^\times\wedge F^\times=0$ in this case. We have also seen 
that  $E^\infty_{1,2}$ has order at most $2$ for any finite field and that this 
term is already $0$ for finite fields of characteristic $2$.  

Thus the convergence of the spectral sequence gives us a complex 
\[
0\to E^\infty_{3,0}\to\ho{3}{\spl{2}{F}}{\zhalf{\Z}}\to E^\infty_{0,3}\to 0.
\]
which is exact except possible at the middle term. If we denote the middle 
 homology group 
by $H(F)$, then it admits  a short exact sequence 
\[
0\to E^\infty_{1,2}\to H(F)\to E^\infty_{2,1}\to 0.
\]

This completes the proof of Theorem \ref{thm:main}.

\section{The refined Bloch group, the classical Bloch group and indecomposable $K_3$}

Recall that for any field $F$ there is a natural homomorphism $\ho{3}{\spl{2}{F}}{\Z}\to\kind{F}$ which 
factors as follows:
\[
\xymatrix{
\ho{3}{\spl{2}{F}}{\Z}\ar[r]
&\ho{3}{\spl{}{F}}{\Z}
&K_3(F)/(\{ -1\}\cdot K_2(F))\ar[l]^-{\cong}\ar@{>>}[r]
&
\kind{F}.
}
\]

Now, for any infinite field $F$ this map is surjective (see \cite{hutchinson:tao2}), and the induced homomorphism
\[
\xymatrix{\ho{3}{\spl{2}{F}}{\Z}_{F^\times}\ar@{>>}[r]&\kind{F}}
\] 
has a $2$-primary torsion kernel (see Mirzaii \cite{mirzaii:third}). 

Suslin, \cite{sus:bloch}, has shown that for any infinite field $F$ there is a natural short exact sequence
\[
0\to\covtor{F}\to\kind{F}\to\bl{F}\to 0
\]
where $\covtor{F}$ denotes the unique nontrivial extension of $\Tor{\mu_F}{\mu_F}$ by $\Z/2$ if the characteristic 
of $F$ is not $2$, and denotes  $\Tor{\mu_F}{\mu_F}$ in characteristic $2$. (We will show that this result extends 
to finite fields in section \ref{sec:blochfinite} below.)

\begin{cor}\label{cor:blochinf}
Let $F$ be an infinite field. Then the natural map $\rbl{F}\to\bl{F}$ is surjective and the induced map 
$\rbl{F}_{F^\times}\to \bl{F}$ has a $2$-primary torsion kernel.  
\end{cor}

\begin{proof} 
 Combining the preceding remarks  with Theorem \ref{thm:main} gives the commutative diagram (defining $K$)
\begin{eqnarray*}
\xymatrix{
0\ar[r]
& 
K\ar[d]\ar[r]
&
\ho{3}{\spl{2}{F}}{\Z}\ar[d]\ar[r]
&
\rbl{F}\ar[r]\ar[d]
&
0\\
0\ar[r]
& 
\covtor{F}\ar[r]
&
\kind{F}\ar[r]
&
\bl{F}\ar[r]
&
0
}
\end{eqnarray*}
from which the first statement follows. Taking $F^\times$-coinvariants of the top row and noting that the 
natural map $K\to\covtor{F}$ has cokernel annihilated by $4$, the second statement also follows. 
\end{proof}

Now for any field $F$ let 
\begin{eqnarray*}
\ho{3}{\spl{2}{F}}{\Z}_0:=\ker{\ho{3}{\spl{2}{F}}{\Z}\to\kind{F}}
\end{eqnarray*}
and
\begin{eqnarray*}
\rbl{F}_0:=\ker{\rbl{F}\to\bl{F}}
\end{eqnarray*}

\begin{lem} \label{lem:h3sl20}
Let $F$ be an infinite field. Then
\begin{enumerate}
\item $\ho{3}{\spl{2}{F}}{\zzhalf{\Z}}_0=\zzhalf{\rbl{F}}_0$
\item $\ho{3}{\spl{2}{F}}{\zzhalf{\Z}}_0=\aug{F}\ho{3}{\spl{2}{F}}{\zzhalf{\Z}}$ and
$\zzhalf{\rbl{F}}_0=\aug{F}\zzhalf{\rbl{F}}$.
\item 
$
\ho{3}{\spl{2}{F}}{\zzhalf{\Z}}_0 = \ker{\ho{3}{\spl{2}{F}}{\zzhalf{\Z}}\to \ho{3}{\spl{3}{F}}{\zzhalf{\Z}}}\\
= \ker{\ho{3}{\spl{2}{F}}{\zzhalf{\Z}}\to \ho{3}{\gl{2}{F}}{\zzhalf{\Z}}}
$
\end{enumerate}
\end{lem}
\begin{proof}
\begin{enumerate}
\item This follows from applying $-\otimes\Z[1/2]$ to the diagram in the proof of Corollary \ref{cor:blochinf} and 
noting that $\zzhalf{K}=\zzhalf{\Tor{\mu_F}{\mu_F}}=\zzhalf{\covtor{F}}$. 
\item By Corollary \ref{cor:blochinf} again,  we have $\zzhalf{\rbl{F}}_{F^\times}=\zzhalf{\bl{F}}$ and by the result 
of Mirzaii mentioned above we have $\ho{3}{\spl{2}{F}}{\zzhalf{\Z}}_{F^\times}=\zzhalf{\kind{F}}$.

 Of course, 
for any $F^\times$-module $M$, we have $\aug{F}\cdot M = \ker{M\to M_{F^\times}}$.
\item
For the first equality, observe first that the stabilization map 
\[
\ho{3}{\spl{2}{F}}{\Z}\to \ho{3}{\spl{3}{F}}{\Z}
\]
factors through $\ho{3}{\spl{2}{F}}{\Z}_{F^\times}$, since, for example, $(F^\times)^2$ acts trivially on 
$\ho{3}{\spl{2}{F}}{\Z}$ while $(F^\times)^3$ acts trivially on 
$\ho{3}{\spl{3}{F}}{\Z}$. From the isomorphism $\ho{3}{\spl{2}{F}}{\zzhalf{\Z}}_{F^\times}\cong\zzhalf{\kind{F}}$ 
it thus follows that
\[
\ker{\ho{3}{\spl{2}{F}}{\zzhalf{\Z}}\to \ho{3}{\spl{3}{F}}{\zzhalf{\Z}}}\subset \ho{3}{\spl{2}{F}}{\zzhalf{\Z}}_0.
\] 

On the other hand, the natural map $\ho{3}{\spl{2}{F}}{\Z}\to \kind{F}$ factors through\\
 $\ho{3}{\spl{3}{F}}{\Z}$, giving us the reverse inclusion.
 
Furthermore, by \cite{mirzaii:third}, the map 
\[
\ho{3}{\gl{2}{F}}{\zzhalf{\Z}}\to \ho{3}{\gl{3}{F}}{\zzhalf{\Z}}=\ho{3}{\gl{}{F}}{\zzhalf{\Z}}
\]
is injective, while the map $\ho{3}{\spl{3}{F}}{\Z}\to\ho{3}{\gl{3}{F}}{\Z}$ is always injective (by the stability 
results in \cite{hutchinson:tao2}). This implies the second equality.
\end{enumerate}
\end{proof}

\begin{cor}\label{cor:char0}
Let $F$ be a field of characteristic other than $2$. Let $\bar{F}$ be an algebraic closure of $F$ and let $\tilde{F}$
be the smallest quadratically closed subfield of $\bar{F}$ containing $F$. 
Then
\[
\ho{3}{\spl{2}{F}}{\zzhalf{\Z}}_0=\ker{\ho{3}{\spl{2}{F}}{\zzhalf{\Z}}\to\ho{3}{\spl{2}{\tilde{F}}}{\zzhalf{\Z}}}.
\]
\end{cor}

\begin{proof} This follows from the fact that the natural map 
\[
\ho{3}{\spl{2}{E}}{\Z}\to \kind{E}
\] 
is an isomorphism, when $E$ is quadratically closed (\cite{sah:discrete3}), 
together with the fact that $\kind{F}$ satsifies 
Galois descent for finite Galois extensions of degree 
relatively prime to the characteristic of the field (Levine \cite{levine:k3ind}, Merkurjev and 
Suslin \cite{sus:merkurjev}).
\end{proof}  

\begin{rem}\label{rem:k3} In \cite{hutchinson:tao2}, it is shown that, for any infinite field $F$, 
\[
\ho{3}{\spl{3}{F}}{\Z}=\ho{3}{\spl{}{F}}{\Z}=\frac{K_3(F)}{\{ -1\}\cdot K_2(F)}.
\]

Thus, it follows  that 
$\ho{3}{\spl{3}{F}}{\zzhalf{\Z}}\cong \zzhalf{K_3(F)}$ 
(since $\{ -1\}\cdot K_2(F)\subset K_3(F)$ is clearly killed by $2$).
Again, in \cite{hutchinson:tao2} it is shown that the 
cokernel of $\ho{3}{\spl{2}{F}}{\Z}\to \ho{3}{\spl{3}{F}}{\Z}$ is 
$2\cdot K_3^M(F)$, while 
the image of this map is isomorphic to $\kind{F}$. 
\end{rem}

\section{The map $\ho{3}{G}{\Z}\to \rbl{F}$ for subgroups $G$ of $\spl{2}{F}$}\label{sec:h3g}

\subsection{Preliminary Remarks}  Let $G$ be a group and let $P$ be a left $G$-set with at least $5$ elements. 
As in section \ref{sec:main}, let 
$L_\bullet$ be the complex of $\Z[G]$-modules defined by $L_n$ is the free abelian group on $(n+1)$-tuples of \emph{distinct} 
points of $P$, and let $d_n:L_n\to L_{n-1}$ be the simplicial boundary map. Thus we have a spectral sequence 
\[
E^1_{p,q}=\ho{p}{G}{L_q}\Longrightarrow \ho{p+q}{G}{L_\bullet}
\]
and $\ho{n}{G}{L_\bullet}=\ho{n}{G}{\Z}$ for $n\leq 3$. 

Thus we have edge homomorphisms 
\[
\xymatrix{\ho{n}{G}{\Z}\ar@{>>}[r]& E^\infty_{0,n}\ar@{^(->}[r]& E^2_{0,n}=H_n\left( (L_\bullet )_G\right)}
\] 

These edge homomorphisms can be constructed as follows:  Let $F_\bullet$ be a (left) projective 
resolution of $\Z$ as a $\Z[G]$-module.  Let $\beta: F_\bullet \to L_\bullet$ be an augmentation-preserving map of complexes
of $\Z[G]$-modules. Then $\beta$ is determined uniquely up to chain homotopy (see, for example, \cite{brown:coh} I.7.4). There is 
an induced map of complexes
\[
\xymatrix{\Z\otimes_{\Z[G]}F_\bullet = (F_\bullet)_G\ar[r]^-{(\beta)_G}&(L_\bullet)_G}
\]  
and, hence, on taking homology, induced maps 
\[
\ho{n}{G}{\Z}=H_n\left( (F_\bullet)_G\right)\to H_n\left( (L_\bullet )_G\right)
\]
(which are independent of the particular chain map $\beta$).

\subsection{Construction of $\beta$}

We will now let $F_\bullet=F_\bullet(G)$ be the homogeneous (left) standard resolution of $\Z$ over $\Z[G]$. 
Thus $F_n$ is the free $\Z$-module on  
$(n+1)$-tuples $(g_0,\ldots,g_n)$ of elements of $G$ and
 $d_n:F_n\to F_{n-1}$ is again the standard simplicial boundary map. $G$ 
acts diagonally on the left on $F_n$. 
So $F_n$ is a free left $\Z[G]$-module with basis consisting of the elements of the form 
$(1,g_1,\ldots,g_n)$.

Now, suppose that $x\in P$ and that 
the orbit of $x$, $G\cdot x$, is not all of $P$ 
(so that $G$ does not act transitively on $P$). Fix $y\in P\setminus G\cdot x$.
In dimensions less than or equal to $3$, we  will use $x$ and $y$ to 
construct a chain map $\beta=\beta^{x,y}:F_\bullet\to L_\bullet$. (It follows, of course, that the 
resulting maps on homology are independent of the choice of $x$ and $y$).

In dimension $0$, we set $\beta_0^{x,y}(g)=g(x)\in P$.

In dimension $1$, we define 
\[
\beta_1^{x,y}(g_0,g_1)=\left\{
\begin{array}{ll}
(g_0(x),g_1(x)),&\mbox{ if }g_0(x)\not= g_1(x)\\
0,&\mbox{ if } g_0(x)=g_1(x)
\end{array}
\right.
\]

In dimension $2$, we define 
\[
\beta_2^{x,y}(g_0,g_1,g_2)=\left\{
\begin{array}{ll}
(g_0(x),g_1(x),g_2(x)),&\mbox{ if }g_0(x),g_1(x),g_2(x)\mbox{ are distinct}\\
0,&\mbox{ if } g_i(x)=g_{i+1}(x)\mbox{ for } i\in \{ 0,1\}\\
(g_0(y),g_0(x),g_1(x))+(g_0(y),g_1(x),g_0(x)),&\mbox{ if } g_0(x)=g_2(x)\not= g_1(x)
\end{array}
\right.
\]

In dimension $3$, we define $\beta_3^{x,y}(g_0,g_1,g_2,g_3)=$ 
\[
\left\{
\begin{array}{ll}
(g_0(x),g_1(x),g_2(x),g_3(x)),&\mbox{ if }g_0(x),\ldots,g_3(x)\mbox{ are distinct}\\
0,&\mbox{ if } g_i(x)=g_{i+1}(x)\mbox{ for } i\in \{ 0,1,2\}\\
0,&\mbox{ if } g_0(x)=g_{2}(x)\mbox{ and } g_1(x)=g_3(x)\\
& \mbox{ and }g_0(y)=g_1(y)\\
(g_0(y),g_1(y),g_0(x),g_1(x))+(g_0(y),g_1(y),g_1(x),g_0(x))&\mbox{ if } g_0(x)=g_{2}(x)\mbox{ and } g_1(x)=g_3(x)\\
& \mbox{ and }g_0(y)\not=g_1(y)\\
(g_0(y),g_0(x),g_1(x),g_3(x))+(g_0(y),g_1(x),g_0(x),g_3(x)),&\mbox{ if } g_0(x)=g_2(x), \mbox{ and } \\
&g_0(x),g_1(x),g_3(x)\mbox{ are distinct}\\
(g_0(x),g_1(y),g_1(x),g_2(x))+(g_0(x),g_1(y),g_2(x),g_1(x)),&\mbox{ if } g_1(x)=g_3(x), \mbox{ and } \\
&g_0(x),g_1(x),g_2(x)\mbox{ are distinct}\\
(g_0(y),g_1(x),g_2(x),g_0(x))-(g_0(y), g_0(x),g_1(x),g_2(x)),&\mbox{ if } g_0(x)=g_3(x)\mbox{ and }\\
&g_0(x),g_1(x),g_2(x)\mbox{ are distinct}\\
\end{array}
\right.
\]

It can be directly verified that these give a well-defined augmentation-preserving  chain map in dimensions 
less than or equal to $3$. 

\subsection{The refined cross ratio map}

We specialize now to the case where $G$ is a subgroup of $\spl{2}{F}$ for some field $F$ and $P=\projl{F}$. 
From the calculations of 
sections \ref{sec:bloch} and \ref{sec:main}, we have 
\[
H_3\left((L_\bullet(\projl{F}))_{\spl{2}{F}}\right)\cong \rpbker{F}\subset \rpb{F}
\]
and the isomorphism is induced by the map 
\[
(L_3(\projl{F}))_{\spl{2}{F}}\to\rpb{F}, (x_0,x_1,x_2,x_3)\mapsto \an{\phi(x_0,x_1,x_2)}\gpb{\frac{\phi(x_0,x_1,x_3)}{\phi(x_0,x_1,x_2)}}
\]

We will call this map the \emph{refined cross ratio} and will denote it by $\rcr$.  
Thus, if $x_0,\ldots,x_3$ are distinct points 
of $\projl{F}$, we have 
\begin{eqnarray*}
\rcr(x_0,x_1,x_2,x_3)=\left\{
\begin{array}{ll}
\an{\frac{(x_2-x_0)(x_0-x_1)}{x_2-x_1}}\gpb{\frac{(x_2-x_1)(x_3-x_0)}{(x_2-x_0)(x_3-x_1)}},
&\mbox{ if } x_i\not=\infty\\
&\\
\an{x_1-x_2}\gpb{\frac{x_1-x_2}{x_1-x_3}},&\mbox{ if } x_0=\infty\\
&\\
\an{x_2-x_0}\gpb{\frac{x_3-x_0}{x_2-x_0}},&\mbox{ if } x_1=\infty\\
&\\
\an{x_0-x_1}\gpb{\frac{x_3-x_0}{x_3-x_1}},&\mbox{ if } x_2=\infty\\
&\\
\an{\frac{(x_2-x_0)(x_0-x_1)}{x_2-x_1}}\gpb{\frac{x_2-x_1}{x_2-x_0}},&\mbox{ if } x_3=\infty\\
\end{array}
\right.
\end{eqnarray*}

Putting all of this together, we conclude:  

If $G$ is a subgroup of $\spl{2}{F}$, then the composite homomorphism 
$\ho{3}{G}{\Z}\to\ho{3}{\spl{2}{F}}{\Z}\to\rpb{F}$ can 
be calculated on the level of chains as 
\[
\xymatrix{
(F_3)_G\ar[r]^-{\beta}
& L_3(\projl{F})_G\ar[r]&
L_3(\projl{F})_{\spl{2}{F}}\ar[r]^-{\rcr}&\rpb{F}
}
\]

\subsection{Finite cyclic subgroups of $\spl{2}{F}$}
We begin with the following observation: Let $G$ be a group and let $F_\bullet$ be the standard (left) homogeneous 
resolution of $\Z$ as a $\Z[G]$-module.  (For convenience below, we will work modulo degenerate simplices; i.e., we 
set $(g_0,\ldots,g_n)=0$ if $g_i=g_{i+1}$ for some $i\leq n-1$.) The augmented resolution
\[
\cdots\to  F_n\to \cdots\to F_0\to \Z=F_{-1}\to 0
\] 
is contractible as a complex of abelian groups via the homotopy $h_n:F_n\to F_{n+1}$ sending 
$(g_0,\ldots,g_n)$ to $(1,g_0,\ldots,g_n)$. 

Thus if $C_\bullet$ is any complex of free $\Z[G]$-modules, with 
$C_0=\Z[G]$ and $C_1\to C_0\to \Z$ the zero map, we can recursively construct an augmentation preserving chain map of 
$\Z[G]$-complexes $\alpha_\bullet:C_\bullet\to F_\bullet$ as follows: We let $\alpha_0=\id{\Z[G]}$, and if 
$e_1,\ldots,e_s$ is a basis of $C_{n+1}$, we set
\[
\alpha_{n+1}(e_i)=h_n(\alpha_n(d e_i)).
\] 

Now $t\in\spl{2}{F}$ have finite order $r$ and let $G=\an{t}$ be the cyclic group generated by $t$, 
and let $C_\bullet$ be the standard 
$2$-periodic resolution of $\Z$ by free $\Z[G]$-modules:
\[
\xymatrix{\cdots \ar[r]^-{t-1}&\Z[G]\ar[r]^-{N}&\Z[G]\ar[r]^{t-1}&\Z[G]\ar@{>>}[r]&\Z}
\]
where $N=1+t+\cdots+t^{r-1}\in\Z[G]$. 

Applying the recipe above to this situation gives a chain map of complexes of $\Z[G]$-modules 
$C_\bullet\to F_\bullet$ which is given in dimension $3$ 
by the formula
\[
\Z[G]=C_3\to F_3,\ 1\mapsto \sum_{i=0}^{r-1}(1,t,t^{i+1},t^{i+2}).
\]

Finally, if we choose $x\in \projl{F}$ and 
if we choose $y\in \projl{F}\setminus G\cdot x$,
then it follows that composite 
\[
\Z/n\cong \ho{3}{G}{\Z}\to \rpb{F}
\]
is given by the formula 
\[
1\mapsto \sum_{i=0}^{r-1}\rcr(\beta_3^{x,y}(1,t,t^{i+1},t^{i+2})).
\]
Furthermore, the uniqueness  up to chain homotopy of the chain map $\beta$ 
guarantees us that the resulting map is independent 
of the particular choice of $x$ and $y$.

If we suppose that $G_x=\{ 1\}$, then, from the definition of $\beta_3^{x,y}$ above, we have:
\begin{eqnarray*}
\beta_3^{x,y}(1,t,t,t^2)&=&0\\
\beta_3^{x,y}(1,t,t^{i+1},t^{i+2})&=&(x,t(x),t^{i+1}(x),t^{i+2}(x))\mbox{ for }1\leq i\leq r-3\\
\beta_3^{x,y}(1,t,t^{r-1},1)&=&(y,t(x),t^{-1}(x),x)-(y, x, t(x),t^{-1}(x))\\
\beta_3^{x,y}(1,t,t^r,t^{r+1})&=&\beta_3^{x,y}(1,t,1,t)\\
&=&\left\{
\begin{array}{ll}
0, & y=t(y)\\
(y,t(y),x,t(x))+(y,t(y),t(x),x),& y\not= t(y)
\end{array}
\right. 
\end{eqnarray*}

For example, if $G_\infty=G\cap B= \{ 1\}$ and if $r\geq 3$, 
then we can take $x=\infty$ and $y\not= t^i(\infty)$ for $i=0,\ldots, n-1$. Suppose also that $y\not= t(y)$. 
Then, using the 
formulae for $\rcr$ given above, we see that  the map
$\Z/n=H_3(G)\to\rpb{F}$ is given by 
\begin{eqnarray*}
1&\mapsto& 
\left(
\sum_{i=1}^{r-3}\an{t(\infty)-t^{i+1}(\infty)}\gpb{\frac{t(\infty)-t^{i+1}(\infty)}{t(\infty)-t^{i+2}(\infty)}}
\right)
\\
&&+
\an{\frac{(t^{-1}(\infty)-y)(y-t(\infty))}{t^{-1}(\infty)-t(\infty))}}
\gpb{\frac{t^{-1}(\infty)-t(\infty)}{t^{-1}(\infty)-y}}-\an{t(\infty)-y}\gpb{\frac{t^{-1}(\infty)-y}
{t(\infty)-y}}
\\
&&+\an{y-t(y)}\gpb{\frac{t(\infty)-y}{t(\infty)-t(y)}}+
\an{\frac{(y-t(y))(t(\infty)-y)}{t(\infty)-t(y)}}\gpb{\frac{t(\infty)-t(y)}{t(\infty)-y}}.
\label{eqn:cyclicgen}
\end{eqnarray*}

\subsection{Third homology of generalised quaternion groups}  

Let $t$ be an even integer and let $Q=Q_{4t}=\langle x,y\ |\ x^t=y^2, xyx=y\rangle$. Again, let $F_\bullet$ be the standard 
(nondegenerate) resolution of $\Z$ over $\Z[Q]$. Let $C_\bullet$ be the $4$-periodic resolution of $\Z$ over $\Z[Q]$ (see 
Cartan-Eilenberg \cite{cartan:eilenberg}, XII.7). We can use the recipe above to construct an augmentation-preserving 
chain map $\alpha_\bullet: C_\bullet \to F_\bullet$. In particular,  $C_3=\Z[Q]$ and we obtain 
\[
\alpha_3(1)=\left(\sum_{i=1}^{t-1}(1,x,x^{i+1},x^{i+2})\right)-(1,x,xy,xy^2)-(1,xy,y^2,xy^2)-(1,xy,y,y^2).
\]

Now $\ho{3}{Q}{\Z}\cong\Z/4t$ and thus the cycle on the right represents a generator of this cyclic group. If 
$q=p^f$ with $p$ an odd prime, then the $2$-Sylow subgroups of $\spl{2}{\F{q}}$ are generalised quaternion and 
we will use the maps $\beta_3$ and $\rcr$ as above to calculate the image $\ho{3}{Q}{\Z}\to\rbl{\F{q}}$.

\section{Bloch groups of finite fields}\label{sec:blochfinite}

In this section we use the calculations of sections \ref{sec:finite} and section \ref{sec:h3g} as well as 
 Theorem \ref{thm:main} to give an explicit description of the Bloch groups of finite fields.

We begin by observing:   

\begin{lem} \label{lem:blochfin}
For a finite field $F$ (with at least $4$ elements) the natural map $\rpb{F}\to\pb{F}$ induces an isomorphism
$\rbl{F}\cong\bl{F}$.
\end{lem}

\begin{proof}
By Lemma \ref{lem:rblfin} we know that $\rbl{F}_{F^\times}\cong\bl{F}$. However, by Lemma \ref{lem:trivial}, 
$F^\times$ acts trivially on $\ho{3}{\spl{2}{F}}{\Z[1/p]}$ (where $p$ is the characteristic of $F$). By Theorem 
\ref{thm:main}, $\rbl{F}$ is  naturally a quotient of the $\sgr{F}$-module $\ho{3}{\spl{2}{F}}{\Z[1/p]}$, and thus 
$F^\times$ acts trivially on $\rbl{F}$. 
\end{proof}

\begin{rem} On the other hand, for a finite field $F$ the map $\rpb{F}\to\pb{F}$ is not an isomorphism 
if the characteristic is odd.  We have a commutative diagram with exact rows
\[
\xymatrix{
0\ar[r]
&
\rbl{F}\ar[r]\ar[d]^-{\cong}
&
\rpb{F}\ar[r]\ar@{>>}[d]
&\aug{F}^2\ar[r]\ar@{>>}[d]
&
0\\
0\ar[r]
&
\bl{F}\ar[r]
&
\pb{F}\ar[r]
&
\sym{2}{\F{2}}{\sq{F}}\ar[r]
&
0
}
\]
from which we derive the short exact sequence
\[
0\to \aug{F}^3\to\rpb{F}\to\pb{F}\to 0.
\]
If the characteristic is odd, then $\aug{F}^3\cong \Z$ with a nontrivial $\sgr{F}$-structure; any nonsquare 
element of $F^\times$ acts as multiplication by $-1$.
\end{rem}

For any field $F$ and for $x\in F^\times$ we define the element $\sus{x}:=\gpb{x}+\gpb{x^{-1}}\in \pb{F}$. The following 
lemma (due to Suslin, \cite{sus:bloch}) is easily verified:

\begin{lem}\label{lem:suslin}
\begin{enumerate}
\item
For any field $F$ (with at least $4$ elements) there is a well-defined group homomorphism 
\[
\sq{F}\to\pb{F},\quad \an{x}\mapsto \sus{x}.
\]
In particular, $\sus{x}=0$ if $x\in (F^\times)^2$, and $2\cdot\sus{y}=0$ for all $y\in F^\times$.
\item For $x\not=1$ let 
\[
\bconst{F}(x)=\gpb{x}+\gpb{1-x}\in\bl{F}
\]
Then $\bconst{F}(x)=\bconst{F}(y)$ for all $x,y\in F^\times\setminus\{ 1\}$, and $3\bconst{F}=\sus{-1}$.
\end{enumerate}
\end{lem}

For a finite field $\F{q}$ of characteristic $2$, 
Theorem \ref{thm:main} tells us that $\rbl{\F{q}}=\bl{\F{q}}=\pb{\F{q}}$ is cyclic of order $q+1$.

If $\F{q}$ has odd characteristic, then Theorem \ref{thm:main} tells us that $\bl{\F{q}}$ is cyclic of 
order $q+1$ or $(q+1)/2$.  In fact, it is always the latter:

\begin{lem}\label{lem:oddfinite}
Let $\F{q}$ be a finite field of odd characteristic. Then $\rbl{\F{q}}=\bl{\F{q}}$ is cyclic of order $(q+1)/2$.
\end{lem} 

\begin{proof} As above, we let 
\[
w=\matr{0}{1}{-1}{0}\in\spl{2}{\F{q}}.
\]

Suppose first that $q\equiv 1\pmod{4}$. Then $\F{q}$ contains an element $i$ satisfying $i^2=-1$. 

Let $2^t$ be the exact power of $2$ dividing $q-1$. 
Then a $2$-Sylow subgroup of $\spl{2}{\F{q}}$ is the  generalised 
quaternion group $Q$ generated by $w$ and $\{ D(z) |z\in \mu_{2^t}\}$. A typical element of $Q$ has the form 
$g=D(z)w^e$ with $e\in \{ 0,1\}$. Clearly, we must prove that the 
composite 
\[
\ho{3}{Q}{\Z}\to\ho{3}{\spl{2}{\F{q}}}{\Z}\to\pb{\F{q}}
\] 
is the zero map. We will calculate using the standard (homogeneous) resolution of $Q$.

Now suppose that $g_k=D(z_k)w^{e(k)}\in Q$,$0\leq k\leq 3$. We will show that 
\[
\rcr(\beta_3^{\infty,i}(g_0,g_1,g_2,g_3))=0\mbox{ in }\pb{\F{q}}.
\]

Since $Q\cdot\infty=\{ 0,\infty\}$, it 
follows 
that either two successive terms of $(g_0(\infty),g_1(\infty),g_2(\infty),g_3(\infty))$ are equal (in which case  
$\beta_3^{\infty,i}(g_0,g_1,g_2,g_3)=0$ or this term has one of the forms 
\[
(\infty,0\infty,0)\mbox{ or }(0,\infty,0,\infty).
\]
Either way, since $w\cdot i =i$, we must have
\[
\beta_3^{\infty,i}(g_0,g_1,g_2,g_3)=(z_0^2i,z_1^2i,\infty,0)+(z_0^2i,z_1^2i,0,\infty)
\]
and applying $\rcr$ to this and taking the image in $\pb{\F{q}}$ gives the element
\[
\gpb{\left(\frac{z_1}{z_0}\right)^2}+\gpb{\left(\frac{z_0}{z_1}\right)^2}=\sus{\left(\frac{z_1}{z_0}\right)^2}=0
\]
by Lemma \ref{lem:suslin}.

On the other hand, if $q\equiv -1\pmod{4}$, we let $G$ be the cyclic subgroup of $\spl{2}{\F{q}}$ generated 
by $w$. Then 
it will be enough to show that the composite
\[
\Z/4\cong\ho{3}{G}{\Z}\to\ho{3}{\spl{2}{\F{q}}}{\Z}\to\rbl{F}
\] 
is the zero map. For the map $\ho{3}{G}{\Z}\to\ho{3}{\spl{2}{\F{q}}}{\Z}$ is injective (by Corollary 
\ref{cor:subgp}) and thus will follow that $\rbl{\F{q}}$ has  order $(q+1)/2$ in this case also.

Using the formulae of section \ref{sec:h3g}, $1\in \Z/4$ maps to the cycle represented by 
\[
(1,w,w,w^2)+(1,w,w^2,w^3)+(1,w,w^3,1)+(1,w,1,w)
\]
in the standard resolution for $G$. Applying $\beta_3^{\infty,y}$ to this gives the term 
\[
2\cdot[(y,w\cdot y,\infty,0)+(y,w\cdot y,0,\infty)].
\]
Finally, applying $\rcr$ to this and taking the image in $\pb{\F{q}}$ gives the element
\[
2\cdot\sus{\frac{w(y)}{y}}=2\sus{-\frac{1}{y^2}}\in\pb{\F{q}}
\]
which is zero by Lemma \ref{lem:suslin} again.
\end{proof}

We can thus extend the main result of \cite{sus:bloch} to the case of finite fields:

\begin{cor}\label{cor:susb} For any finite field $F$ there is a natural short exact sequence
\[
0\to\covtor{F}\to\kind{F}\to\bl{F}\to 0
\]
\end{cor}

\begin{proof}
By Corollary \ref{cor:1/p},  Theorem \ref{thm:main} and Lemma \ref{lem:oddfinite} we have a short exact
sequence 
\[
0\to\covtor{\F{q}}\to\ho{3}{\spl{2}{\F{q}}}{\Z[1/p]}\to\rbl{\F{q}}\to 0
\]
for any finite field $\F{q}$ of order $q=p^f$. However, by Corollary \ref{cor:k3} and Lemma \ref{lem:blochfin}
we have natural isomorphisms
\[
\ho{3}{\spl{2}{\F{q}}}{\Z[1/p]}\cong\kind{\F{q}}\mbox{ and }\rbl{\F{q}}\cong\bl{\F{q}}.
\]
\end{proof}

\begin{cor}
If $q\equiv 1\pmod{4}$ then $\pb{\F{q}}$ is cyclic of order $q+1$.
\end{cor}
\begin{proof}
In this case $(q+1)/2$ is odd. Since $\sym{2}{\F{2}}{\sq{F}}$ has order $2$, the statement follows from Lemma 
\ref{lem:oddfinite} and the short exact sequence
\[
0\to\bl{\F{q}}\to\pb{\F{q}}\to\sym{2}{\F{2}}{\sq{F}}\to 0.
\] 
\end{proof}

In fact we can use the methods of the last section to write down a formula for a generator of this cyclic group:
Fix a nonsquare element $a\in \F{q}^\times$. Thus $\F{q^2}=\F{q}(\sqrt{a})$ and we have an associated embedding 
\[
\mu: \F{q^2}^\times\to \gl{2}{\F{q}},\qquad x+y\sqrt{a}\mapsto \matr{x}{yb}{y}{x}.
\] 
Now let $\theta\in \F{q^2}$ have order $r:=(q+1)/2$. Then $t=\mu(\theta)\in\spl{2}{\F{q}}$. Let $G=\an{t}\subset
\spl{2}{\F{q}}$. The results above guarantee that the composite homomorphism
\[
\ho{3}{G}{\Z}\to\ho{3}{\spl{2}{\F{q}}}{\Z}\to\bl{\F{q}}
\]
is an isomorphism. 

Since $B\subset \spl{2}{\F{q}}$ has order $q(q-1)$, it follows that $G\cap B=\{ 1\}$. It follows that 
the orbit $G\cdot\infty$ has size $(q+1)/2$. Choosing any $y\in\projl{\F{q}}\setminus G\cdot\infty$ we obtain 
a generator
\begin{eqnarray*}
\left(
\sum_{i=1}^{r-3}\gpb{\frac{t(\infty)-t^{i+1}(\infty)}{t(\infty)-t^{i+2}(\infty)}}
\right)
+
\gpb{\frac{t^{-1}(\infty)-t(\infty)}{t^{-1}(\infty)-y}}-\gpb{\frac{t^{-1}(\infty)-y}{t(\infty)-y}}
+\gpb{\frac{t(\infty)-y}{t(\infty)-t(y)}}+\gpb{\frac{t(\infty)-t(y)}{t(\infty)-y}}
\label{eqn:cyclicgen}
\end{eqnarray*}
of $\bl{\F{q}}$. The last four terms can be simplified: Observe that if $\theta=w+z\sqrt{a}$ then  $\theta^{-1}=
w-z\sqrt{a}$, since $\theta$ has norm $1$. Thus $t(\infty)=w/z$ , $t^{-1}(\infty)=-w/z$ and $t(y)=(wy+az)/(zy+w)$.  
If we let $A=(t^{-1}(\infty)-y)/(t(\infty)-y)$, then we have
\begin{eqnarray*}
\gpb{\frac{t^{-1}(\infty)-t(\infty)}{t^{-1}(\infty)-y}}-\gpb{\frac{t^{-1}(\infty)-y}{t(\infty)-y}}
+\gpb{\frac{t(\infty)-y}{t(\infty)-t(y)}}+\gpb{\frac{t(\infty)-t(y)}{t(\infty)-y}}
&=&
\bconst{\F{q}}(A^{-1})-\sus{A}+\sus{\frac{t(\infty)-t(y)}{t(\infty)-y}}\\
=\bconst{\F{q}}+\sus{\frac{t^{-1}(\infty)-y}{t(\infty)-t(y)}}&=&\bconst{\F{q}}+\sus{-1}=\bconst{\F{q}}
\end{eqnarray*}  
So when $q\equiv 1\pmod{4}$ a generator of $\bl{\F{q}}$ is 
\[
 \left(
\sum_{i=1}^{r-3}\gpb{\frac{t(\infty)-t^{i+1}(\infty)}{t(\infty)-t^{i+2}(\infty)}}
\right)
+\bconst{\F{q}}.
\]
On the other hand, note that, since 
\[
\lambda(\sus{a})= (1-a)\asymm a + (1-a^{-1})\asymm a^{-1}=a\asymm a\in \sym{2}{\F{2}}{\sq{F}}
\]
it follows that $\sus{a}\in\pb{\F{q}}$ has order $2$ and thus 
\[
H_\theta:= \left(
\sum_{i=1}^{r-3}\gpb{\frac{t(\infty)-t^{i+1}(\infty)}{t(\infty)-t^{i+2}(\infty)}}
\right)
+\bconst{\F{q}}+\sus{a}
\]
is a generator of the cyclic group $\pb{\F{q}}$. 

\begin{rem}
If we let $\rr{\theta}$ be the corresponding isomorphism
\[
\rr{\theta}:\Z/(q+1)\to \pb{\F{q}},\qquad m\mapsto m H_\theta
\]
then the inverse map $\dilog{\theta}:\pb{\F{q}}\to\Z/(q+1)$  is a `universal discrete dilogarithm' on $\F{q}$ 
in the following 
sense: If $A$ is an (additive) abelian group and if $L:\F{q}^\times\to A$ is any map of sets satisfying $L(1)=0$ and 
\[
L(x)-L(y)+L\left(\frac{y}{x}\right)-L\left(\frac{1-x^{-1}}{1-y^{-1}}\right)+L\left(\frac{1-x}{1-y}\right)=0
\mbox{ for all } x,y\in \F{q}\setminus\{ 0,1\}
\]  
then there is a unique homomorphism $\tau:\Z/(q+1)\to A$ such that 
\[
L(x)=\tau(\dilog{\theta}(\gpb{x}))\mbox{ for all } x\in \F{q}^\times.
\]
\end{rem}

When $q\equiv -1\pmod{4}$, we can similarly obtain a formula for a generator of $\bl{\F{q}}$, but in this case 
we must compute a (more complicated) homomorphism $\ho{3}{Q}{\Z}\to\bl{\F{q}}$ where $Q$ is a 
generalised quaternion subgroup of $\spl{2}{\F{q}}$. As an example of a related calculation we prove
\begin{lem}\label{lem:sus-1}
Suppose that $q\equiv-1\pmod{4}$.  Then $\sus{-1}$ has order $2$ in $\bl{\F{q}}$. 
\end{lem}  
\begin{proof}
The calculations above allow us to conclude that $\spl{2}{\F{q}}$ contains a quaternion subgroup $Q$ of order $8$ 
with the property that the composite map
\[
\Z/8\cong\ho{3}{Q}{\Z}\to\ho{3}{\spl{2}{\F{q}}}{\Z}\to\bl{\F{q}}
\]
has image of order $2$. Now we can take generators $x$ and $y$ of $Q$ satisfying
\[
x=w=\matr{0}{1}{-1}{0},\qquad y^2=x^2=-I,\qquad xyx=y.
\]
By the calculations of the last section, a generator of $\ho{3}{Q}{\Z}$ is represented by the cycle
\[
(1,x,x^2,x^3)-(1,x,xy,xy^2)-(1,xy,y^2,xy^2)-(1,xy,y,y^2).
\]
Let $a:=y\cdot\infty$. Then $x\cdot\infty=0$ and $(xy)\cdot\infty=x\cdot a=-a^{-1}$. Choose 
$z\in\projl{\F{q}}\setminus\{\infty,0,a,-1/a\}$. Applying $\beta_3^{\infty,z}$ to this cycle gives the 
element
\begin{eqnarray*}
\left[\left(z,-1/z,\infty,0\right)+\left( z,-1/z,0,\infty\right)\right]
-\left[\left(\infty,-1/z,0,a\right)+\left(\infty,-1/z,a,0\right)\right]\\
-\left[\left(z,\infty,a,0\right)+(z,a,\infty,0)\right]
-\left[\left(z,a,-1/a,\infty\right)-\left(z,\infty,a,-1/a\right)\right].
\end{eqnarray*}
Applying $\rcr$ to this gives the element 
\[
X:=\sus{-z^2}-\sus{1+az}-\gpb{\frac{z}{z-a}}-\gpb{\frac{z}{a}}-\gpb{\frac{1+a^2}{1+az}}
+\gpb{\frac{1+az}{a(z-a)}}\in\bl{\F{q}}.
\]
We easily verify that
\[
-\gpb{\frac{z}{z-a}}-\gpb{\frac{z}{a}}=\bconst{\F{q}}-\sus{\frac{z}{z-a}}-\sus{\frac{z}{a}}
\]
and
\[
-\gpb{\frac{1+a^2}{1+az}}+\gpb{\frac{1+az}{a(z-a)}}=\sus{\frac{1+az}{a(z-a}}-\bconst{\F{q}}
\]
from which it follows that $X=\sus{-1}$. 
\end{proof}
\begin{cor} If $q\equiv 3\pmod{8}$, then $\pb{\F{q}}\cong \Z/(q+1)$.
\end{cor}
\begin{proof} Since $\sus{-1}=2\gpb{-1}$ in $\pb{\F{q}}$, the computation just completed shows that
$\gpb{-1}$ has order $4$ in $\pb{\F{q}}$ when $q\equiv 3\pmod{4}$. On the other hand, if 
$q\equiv 3\pmod{8}$, then the $2$-Sylow subgroup of $\pb{\F{q}}$ has order $4$.  
\end{proof}
\begin{rem} On the other hand, if $q\equiv 7\pmod{8}$ then $2$ is a square in $\F{q}^\times$ and hence 
$\gpb{-1}\in\bl{\F{q}}$ has order $4$. In particular, if $q\equiv 7\pmod{16}$, then $\gpb{-1}$ generates the 
$2$-Sylow subgroup of $\bl{\F{q}}$.
\end{rem}


Finally, for any prime power $q$ let 
\[
t=\matr{0}{1}{-1}{-1}\in\spl{2}{\F{q}}
\]
of order $3$ and let $G=\an{t}\subset\spl{2}{\F{q}}$. Then the composite 
\[
\Z/3=\ho{3}{G}{\Z}\to\ho{3}{\spl{2}{\F{q}}}{\Z}\to\bl{\F{q}}
\]
sends $1$ to $\bconst{\F{q}}+\sus{-1}=4\bconst{\F{q}}$. Thus, this element has order $3$ if $3$ divides $q+1$ 
and (of course) has order $1$ otherwise. In view of Lemma \ref{lem:sus-1} we deduce
\begin{lem} The order of $\bconst{\F{q}}\in\bl{\F{q}}$ is $\mathrm{gcd}(6,(q+1)/2)$.
\end{lem}

\bibliography{BWSL2}
\end{document}